\documentclass[12pt, reqno]{amsart}
\usepackage{amsmath}
\usepackage{amssymb}
\usepackage{epsfig}
\usepackage{graphicx}
\usepackage{enumitem}
\usepackage{color}
\usepackage{fullpage}
\definecolor{shadecolor}{gray}{0.875}
\usepackage{amscd}
\usepackage{comment}
\usepackage{adjustbox}
\usepackage{multirow}

\numberwithin{equation}{section}

\newcommand{\brian}[1]{{\color{blue} \sf $\clubsuit\clubsuit\clubsuit$ Brian: [#1]}}
\newcommand{\sho}[1]{{\color{red} \sf $\clubsuit\clubsuit\clubsuit$ Sho: [#1]}}
\newcommand{\eric}[1]{{\color{cyan} \sf $\clubsuit\clubsuit\clubsuit$ Eric: [#1]}}

\input xy
\xyoption{all}

\calclayout
\allowdisplaybreaks[3]

\theoremstyle{plain}
\newtheorem{prop}{Proposition}[section]

\newtheorem{theo}[prop]{Theorem}
\newtheorem{coro}[prop]{Corollary}

\newtheorem{lemm}[prop]{Lemma}

\theoremstyle{definition}
\newtheorem{defi}[prop]{Definition}

\newtheorem{conj}[prop]{Conjecture}

\newtheorem{rema}[prop]{Remark}

\newtheorem{exam}[prop]{Example}

\newtheorem{clai}[prop]{Claim}

\newcounter{enumi_saved}
\def\cO{{\mathcal O}}

\def\cU{{\mathcal U}}

\def\bC{{\mathbb C}}

\def\bP{{\mathbb P}}

\def\Eff{\overline{\mathrm{Eff}}}

\def\Nef{\mathrm{Nef}}

\def\Span{\mathrm{Span}}

\def\Sym{\mathrm{Sym}}

\def\codim{\mathrm{codim}}

\def\Rat{\operatorname{Rat}}

\makeatother
\makeatletter

\author{Roya Beheshti}
\address{Department of Mathematics \\
Washington University in St.~Louis \\
St. Louis, MO \, \, 63130}
\email{beheshti@wustl.edu}

\author{Brian Lehmann}
\address{Department of Mathematics \\
Boston College  \\
Chestnut Hill, MA \, \, 02467}
\email{lehmannb@bc.edu}

\author{Eric Riedl}
\address{Department of Mathematics \\
University of Notre Dame  \\
255 Hurley Hall \\
Notre Dame, IN 46556}
\email{eriedl@nd.edu}

\author{Sho Tanimoto}
\address{Graduate School of Mathematics, Nagoya University, Furocho Chikusa-ku, Nagoya, 464-8602, Japan}
\email{sho.tanimoto@math.nagoya-u.ac.jp}

\title[Moduli spaces of rational curves on Fano threefolds]{Moduli spaces of rational curves on Fano threefolds}

\begin{document}

\begin{abstract}
We prove several classification results for the components of the moduli space of rational curves on a smooth Fano threefold. In particular, we prove a conjecture of Batyrev on the growth of the number of components as the degree increases. The key to our approach is Geometric Manin's Conjecture which predicts the number of components parameterizing free curves.
\end{abstract}

\maketitle

\section{Introduction}
While lines and conics on Fano threefolds have been extensively studied due to their role in the classification theory, much less is known about rational curves of higher degree.   Let $X$ denote a smooth Fano threefold and let $\overline{\Rat}(X)$ denote the union of the components of $\overline{M}_{0,0}(X)$ that generically parametrize stable maps with irreducible domain. An important open question is to classify all components of $\overline{\Rat}(X)$.

In this paper we develop a theoretical framework which reduces the classification problem to understanding the components in a finite degree range. As an application, we prove a conjecture of Batyrev that the number of components of the space of rational curves on a Fano threefold is bounded by a polynomial in the degree of the curve.

\begin{theo} \label{theo:maintheorem4}
Let $X$ be a smooth Fano threefold.  There is a polynomial $P(d)$ such that the number of components of $\overline{\Rat}(X)$ parameterizing rational curves of anticanonical degree $\leq d$ is bounded above by $P(d)$.
\end{theo}

Our approach to the classification problem is motivated by Geometric Manin's Conjecture.  Just as Manin's Conjecture for rational points predicts the asymptotic growth rate of the number of rational points of bounded height, Geometric Manin's Conjecture predicts the growth rate of the dimension and number of components of $\overline{\Rat}(X)$ parameterizing degree $d$ curves as the degree $d$ grows large.  By analogy with the rational point version, we must distinguish two types of components: the \emph{accumulating components} which do not contribute to the counting function, and the \emph{Manin components} which do contribute.

\subsection{Accumulating components}
There are two key perspectives on how to identify accumulating components for Fano threefolds.  First, \cite{LST18} gives a conjectural description of the exceptional set in Manin's Conjecture based on the negativity properties of the canonical divisor.  Second, one can identify accumulating components based on their  pathological geometric behavior -- for example, they could have higher than the expected dimension or the evaluation map on the universal family could fail to have connected fibers.

In Theorem \ref{theo:classifymanincomponents} and Theorem \ref{theo:pathologicalisnegligible} we will show that these two approaches match up perfectly (with one possible exception). In particular, using techniques from the Minimal Model Program we can explicitly identify all families of curves which exhibit pathological geometric behavior.  Our first result identifies the families which have larger than expected dimension, improving the implicit description of \cite{LTCompos}:

\begin{theo} \label{theo:maintheorem1}
Let $X$ be a smooth Fano threefold.  Let $M$ be a component of $\overline{\Rat}(X)$ such that the curves parametrized by $M$ sweep out a proper subvariety $Y \subsetneq X$.  Then either:
\begin{itemize}
\item $Y$ is swept out by a family of $-K_{X}$-lines, or
\item $Y$ is an exceptional divisor for a birational contraction on $X$.
\end{itemize}
\end{theo}

Our second result identifies the dominant families of rational curves such that the evaluation map for the universal family fails to have connected fibers. 

\begin{theo} \label{theo:maintheorem2}
Let $X$ be a smooth Fano threefold such that $-K_X$ is very ample. Let $M$ be a component of $\overline{\Rat}(X)$ and let $\mathcal{C} \to M$ denote the corresponding component of $\overline{M}_{0,1}(X)$.  Suppose that the evaluation map $ev: \mathcal{C} \to X$ is dominant but a general fiber is not irreducible.  Then either:
\begin{itemize}
\item $M$ parametrizes a family of stable maps whose images are $-K_{X}$-conics, or 
\item $M$ parametrizes a family of curves contracted by a del Pezzo fibration $\pi: X \to Z$.
\end{itemize}
\end{theo}

\subsection{Manin components}

Manin components are the components of $\overline{\Rat}(X)$ which are included in the counting function.  Our main tool for counting Manin components is the following result.

\begin{theo}[Movable Bend-and-Break] \label{theo:maintheorem3}
Let $X$ be a smooth Fano threefold.  Let $M$ be a component of $\overline{\Rat}(X)$ that generically parametrizes free curves.  Suppose that the general curve $C$ parametrized by $M$ has anticanonical degree $\geq 6$.  Then $M$ contains a stable map of the form $f: Z \to X$ where $Z$ has two components and the restriction of $f$ to each component realizes this component as a free curve on $X$.
\end{theo}

Using this theorem, we can apply the inductive approach pioneered by \cite{HRS04} to classify free curves on $X$.  Suppose we identify the (finitely many) components of $\overline{\Rat}(X)$ which parametrize free curves of anticanonical degree $\leq 5$.  Then we can classify all components of $\overline{\Rat}(X)$ parametrizing free curves of anticanonical degree $\geq 6$ by repeatedly gluing the low degree free curves and studying the resulting components. 

Note that the gluing and smoothing operation defines a binary product on the set of components of $\overline{\Rat}(X)$ whose evaluation maps are dominant and have connected fibers.  In many situations one can use Movable Bend-and-Break (Theorem~\ref{theo:maintheorem3}) to show that this monoid is finitely generated.  In this way the irreducibility of the moduli space of rational curves associated to each nef class can be reduced to showing the irreducibility of moduli spaces in low degree.  We demonstrate this technique in a couple examples in Section \ref{sec:examples}.

\begin{exam}
Let $X \subset \mathbb{P}^{4}$ be a smooth quartic threefold.  In Section \ref{sect:quarticthreefold} we will show that in every degree $d \geq 3$ there is a unique component of $\overline{\Rat}(X)$ of degree $d$ that parametrizes birational maps onto free curves.  All other components of $\overline{\Rat}(X)$ with $d \geq 3$ parametrize multiple covers of lines or conics.  (Note that our results do not address the components of $\overline{M}_{0,0}(X)$ which are not contained in $\overline{\Rat}(X)$.)
When $X$ is general, such a statement has been obtained in \cite{LTJAG}.
\end{exam}

\begin{exam}
Let $X$ be the blow-up of $\mathbb{P}_{\mathbb{P}^{2}}(\mathcal{O} \oplus \mathcal{O}(2))$ along a quartic curve in a minimal moving section.  In Section \ref{sect:2E5contractions} we will show that every nef curve class on $X$ is represented by a unique component of $\overline{\Rat}(X)$ and that every pseudo-effective curve class on $X$ is represented by at most one component of $\overline{\Rat}(X)$.  Since $X$ is the unique Fano threefold which admits two E5 contractions, it is in some sense the ``most difficult'' Fano threefold to handle using our techniques.
\end{exam}

\subsection{Comparison to previous work}
There are some examples of Fano threefolds for which the components of $\overline{\Rat}(X)$ were calculated previously.  Lines and conics on Fano threefolds are well studied due to their relations to the classification theory, and there is an extensive literature initiated by works of Iskovskikh \cite{Isk77}, \cite{FanoII}, and \cite{Isk79}. See \cite{iskov} and \cite{KPS18} for more information and references. For rational curves of higher degree, cubic threefolds were addressed in \cite{CS09} and intersections of two quadrics were handled by \cite{Cas04}.  \cite{LTCompos} and \cite{LTJAG} proved some partial results for Fano threefolds of Picard rank $1$.  There are more general results for homogeneous varieties (\cite{Thomsen98}, \cite{KP01}) and for toric varieties (\cite{Bourqui12}) which can be applied to Fano threefolds of the appropriate type.  

There are two technical advances which distinguish the results of this paper from previous work.  First, we fully describe the behavior of the $a$ and $b$ invariants for Fano threefolds with $-K_X$ very ample. By appealing to the Minimal Model Program for threefolds and the classification results of \cite{Kawakita05}, we give a complete classification of the generically finite maps $f: Y \to X$ such that $a(Y,-f^{*}K_{X}) \geq a(X,-K_{X})$ and $-K_X$ is very ample, finishing a project initiated in \cite{LTT18} and \cite{LT16}.  This classification allows us to leverage the theory of Geometric Manin's Conjecture to classify pathological components of $\overline{\Rat}(X)$.  (Our work can also be used to describe the conjectural exceptional set in Manin's Conjecture for a Fano threefold defined over a number field which is proposed in \cite{LST18}.)

Second, we give a general and conceptual proof of Movable Bend-and-Break for Fano threefolds.  In particular our work clarifies and greatly extends the ad hoc approach used to prove a special case in \cite{LTJAG}.  To prove Movable Bend-and-Break for a family of free curves, the key idea is to impose as many point and curve incidences as possible on a one-dimensional subfamily.  Using the geometry of these incidence correspondences we are able to constrain the possible outcomes of Mori's Bend-and-Break Lemma.

$ $ \\

\noindent
{\bf Acknowledgements:}
The authors thank Kento Fujita for information on Mori cones of curves on Fano threefolds. The authors also would like to thank Masayuki Kawakita for answering our questions regarding terminal divisorial contractions. The authors thank Kento Fujita, Masayuki Kawakita, and Tony V\'arilly-Alvarado for their interest in this work. The authors also would like to thank Eric Jovinelly for pointing out a mistake in an earlier version of this paper.

This project was started at the SQuaRE workshop ``Geometric Manin's Conjecture in characteristic $p$'' at the American Institute of Mathematics. The authors would like to thank AIM for the generous support.

Brian Lehmann was supported by NSF grant 1600875. Eric Riedl was supported by NSF grant 1945144.  Sho Tanimoto was partially supported by Inamori Foundation, by JSPS KAKENHI Early-Career Scientists Grant number 19K14512, by JSPS Bilateral Joint Research Projects Grant number JPJSBP120219935, by MEXT Japan, Leading Initiative for Excellent Young Researchers (LEADER), and by JST FOREST program Grant number JPMJFR212Z.

\section{Preliminaries}

We work over an algebraically closed field $k$ of characteristic $0$.
A variety is an integral separated scheme of finite type over $k$.
Let $X$ be a projective variety defined over $k$. Let $N^1(X)$ be the space of $\mathbb R$-Cartier divisors modulo numerical equivalence and let $N_1(X)$ be the dual space of $\mathbb R$-1-cycles modulo numerical equivalence. We denote the nef cone and the pseudo-effective cone of divisors by
\[
\Nef^1(X), \quad \Eff^1(X)
\]
and we denote the nef cone and the pseudo-effective cone of curves by
\[
\Nef_1(X), \quad \Eff_1(X).
\]

\subsection{Mori's Bend-and-Break Lemma}

We will use the following version of Mori's Bend-and-Break Lemma.

\begin{lemm}[\cite{LT19} Lemma 4.1]\label{bend-and-break}
\label{lemm:BandB}
Let $\pi : S\to  C$ be a morphism from a smooth projective surface to a smooth projective curve such that a general fiber of $\pi$ is isomorphic to $\mathbb P^1$. Suppose that we have a morphism $f : S \to X$ to a projective variety such that (i) the image of $f$ is $2$-dimensional, and (ii) there are two sections $C_1$, $C_2$ of $\pi$ contracted to distinct points $x_1$ and $x_2$ on $X$ respectively. Then there exists a singular fiber $F$ of $\pi$ and two components $F_1$ and $F_2$ of $F$ such that (a) both $F_1$ and $F_2$ are not contracted by $f$, and (b) the image of $F_i$ contains $x_i$.
\end{lemm}

\subsection{Classification results}
\label{subsec:classification}

For later analysis, we need the classification of divisorial contractions for smooth projective threefolds
developed by Mori in \cite{Mori82}:

\begin{theo}[\cite{Mori82}] \label{theo:moriclassification}
Let $X$ be a smooth projective threefold and $f : X \rightarrow Y$ be a $K_X$-negative divisorial contraction. Let $E$ be the exceptional divisor on $X$. Then $f$ and $E$ are described by one of following list of possibilities:
\begin{itemize}
\item E1: $f$ is the blow up along a smooth projective curve in a smooth projective threefold $Y$. The divisor $E$ is a ruled surface and any fiber $C$ of $f|_{E}$ satisfies $-K_X \cdot C = 1$;
\item E2: $f$ is the blow up at a smooth point of a smooth projective threefold $Y$. The polarized surface $(E, -K_X|_E)$ is isomorphic to $(\mathbb P^2, \mathcal O(2))$;
\item E3: $f$ is the blow up at an ordinary double point of $Y$ which is locally analytically isomorphic to $x^2 + y^2 + z^2 + w^2 = 0$ in $\mathbb A^4$. The polarized surface $(E, -K_X|_E)$ is isomorphic to $(Q, \mathcal O(1, 1))$ where $Q$ is a smooth quadric surface;
\item E4: $f$ is the blow up of a point in $Y$ which is locally analytically isomorphic to $x^2 + y^2 + z^2 + w^3 = 0$. The polarized surface $(E, -K_X|_E)$ is isomorphic to $(Q, \mathcal O(1))$ where $Q$ is a quadric cone;
\item E5: $f$ is the blow up of a point in $Y$ which is locally analytically isomorphic to the quotient of $\mathbb A^3$ by the involution $(x, y, z) \mapsto (-x, -y, -z)$. The polarized surface $(E, -K_X|_E)$ is isomorphic to $(\mathbb P^2, \mathcal O(1))$.
\end{itemize}
\end{theo}

Given a contractible divisor $E$ on a smooth Fano threefold $X$, we will say that $E$ has type E1, E2, E3, E4, or E5 according to the classification above.  The following lemma allows us to easily work with contractible divisors.

\begin{lemm} \label{lemm:nooverlap}
Let $X$ be a smooth Fano threefold.  Let $E_{1}$ and $E_{2}$ be distinct contractible divisors which have E2, E3, E4, E5 type or which have E1 type and are isomorphic to $\mathbb{P}^{1} \times \mathbb{P}^{1}$ with normal bundle $\mathcal{O}(-1,-1)$.  Then $E_{1} \cap E_{2} = \varnothing$.
\end{lemm}

The proof is well-known; for example, in \cite{MM83} it is proved for primitive Fano threefolds.

\begin{proof}
Suppose that $E_{1} \cap E_{2}$ is not empty.  Then the intersection must have dimension $1$.  Let $C$ denote an irreducible curve lying in the intersection.  Since both divisors have the property that every curve in the surface deforms in a dominant family, we see that $C$ deforms to sweep out $E_{1}$ and also deforms to sweep out $E_{2}$.  This implies that $C$ is a nef curve class on $X$.  However, we should also have $E_{1} \cdot C < 0$ due to the antiampleness of $E_{1}|_{E_{1}}$, giving a contradiction.
\end{proof}

Smooth Fano threefolds of Picard rank $1$ are classified by Fano-Iskovskih in \cite{Isk77} and \cite{FanoII}. The classification of smooth Fano threefolds of Picard rank $\geq 2$ is given by Mori-Mukai in \cite{MM81}, \cite{MM83}, and \cite{MM03}.  Finally, the computation of extremal rays for smooth Fano threefolds has been carried out in \cite{Mat95}, \cite{MM04}, and \cite[Section 10.4]{Fuj16}.  According to these results, the following list classifies all smooth Fano threefolds which admit an E5 type divisorial contraction. 

\begin{theo} \label{theo:e5classification}
Let $X$ be a smooth Fano threefold admitting an E5 type divisorial contraction.
Then $X$ is one of the following threefolds:
\begin{itemize}
\item Picard rank $2$ case:
\begin{enumerate}
\item the blow up of $\mathbb{P}^{3}$ along a smooth cubic plane curve;
\item $\mathbb{P}_{\mathbb{P}^{2}}(\mathcal{O} \oplus \mathcal{O}(2))$.
\setcounter{enumi_saved}{\value{enumi}}
\end{enumerate}
\item Picard rank $3$ case
\begin{enumerate}
\setcounter{enumi}{\value{enumi_saved}}
\item the blow up of $\mathbb P^3$ along the disjoint union of a plane cubic in a plane $P$ and a point not on $P$;
\item the blow up of $\mathbb{P}^{1} \times \mathbb{P}^{2}$ along a conic in a fiber of the first projection map;
\item the blow up of $\mathrm{Bl}_{\mathrm{pt}}\mathbb{P}^{3}$ along a line contained in the exceptional divisor;
\item the blow-up of $\mathbb{P}_{\mathbb{P}^{2}}(\mathcal{O} \oplus \mathcal{O}(2))$ along a quartic curve $Z$ contained in a minimal moving section $D$ of the projective bundle.
\end{enumerate}
\end{itemize}
Moreover the first five varieties admit a unique E5 divisorial contraction while the last one admits exactly two.
\end{theo}

We will also need the following results about Fano threefolds which follows directly from the classification.

\begin{theo}[\cite{Isk77}] \label{theo:notbpf}
Let $X$ be a smooth Fano threefold such that $|-K_X|$ is not base point free.
Then $X$ is isomorphic to one of the following threefolds:
\begin{enumerate}
\item $V_1$ which is the double cover of the Veronese cone $W_4 \subset \mathbb P^6$ whose branch locus is a smooth intersection of $W_4$ and a cubic hypersurface not passing through the vertex of the cone;
\item the blow up of $V_1$ along a smooth elliptic curve which is a complete intersection of two members of $|-\frac{1}{2}K_{V_1}|$;
\item $\mathbb P^1 \times S_1$ where $S_1$ is a del Pezzo surface of degree $1$.
\end{enumerate}
\end{theo}

\begin{theo}[\cite{iskov}]
\label{theo:notveryample}
Let $X$ be a smooth Fano threefold such that $|-K_X|$ is base point free, but not very ample.
Then $X$ is isomorphic to one of the following threefolds:
\begin{enumerate}
\item the double cover of $\mathbb P^3$ ramified along a smooth sextic surface;
\item the double cover of a quadric hypersurface $Q \subset \mathbb P^4$ ramified along a smooth intersection of $Q$ and a quartic hypersurface;
\item the double cover of $\mathbb P^1 \times \mathbb P^2$ ramified along a smooth hypersurface of bidegree $(2, 4)$;
\item the blow up of $V_2$ along an elliptic curve which is a complete intersection of two members of $|-\frac{1}{2}K_{V_2}|$, and;
\item $\mathbb P^1 \times S_2$ where $S_2$ is a smooth del Pezzo surface of degree $2$. 
\end{enumerate}
\end{theo}

\subsection{Deformation theory for stable maps of genus $0$} \label{sect:deftheory}

Let $X$ be a smooth projective variety.  We denote the coarse moduli space of stable maps of genus $0$ on $X$ by $\overline{M}_{0,0}(X)$.  In this section we will use deformation theory to study $\overline{M}_{0,0}(X)$ following \cite{BM96}, \cite{BF97}, and \cite{GHS03}.  We will be interested in the components of $\overline{M}_{0,0}(X)$ which generically parametrize maps with irreducible domain; we denote the union of such components by $\overline{\Rat}(X)$.
 
Let $f : Z \to X$ be a stable map of genus $0$.
Assume that $f$ is an immersion on an open neighborhood of each nodal point of $Z$. Then the deformation theory of $f$ is controlled by the normal sheaf $N_{f/X}$ which fits into an exact sequence
\[
0 \to \mathcal{H}om_{\mathcal{O}_{C}}(K,\mathcal{O}_{C}) \to N_{f/X} \to \mathcal{E}xt^{1}_{\mathcal{O}_{C}}(Q,\mathcal{O}_{C}) \to 0
\]
where $K$ and $Q$ are respectively the kernel and cokernel of $f^{*}\Omega^{1}_{X} \to \Omega^{1}_{C}$.   Precisely, the space of first-order deformations of $f$ is $H^{0}(Z,N_{f/X})$ and the obstruction space is $H^{1}(Z,N_{f/X})$.  We will often be in a situation where $f$ is an immersion so that $Q=0$.

\begin{prop}[\cite{GHS03} Lemma 2.6]\label{prop:GHS}
Let $X$ be a smooth projective variety.  Suppose that $f: Z \to X$ is a stable map which is an immersion on an open neighborhood of each node of $Z$.  
Suppose that $Z_0$ is an irreducible component of $Z$ which meets components $Z_1, \dots, Z_c$ at points $p_1, \dots, p_c$.  Letting $f_{0} := f|_{Z_{0}}$ we obtain a short exact sequence
\[ 0 \to N_{f_0/X} \to N_{f/X}|_{Z_0} \to \bigoplus_i k(p_i) \to 0 , \]
where the quotient $N_{f/X}|_{Z_0} \to k(p_i)$ is given by the tangent direction of $Z_i$ at $p_i$.
\end{prop}

We will mainly be interested in stable maps whose normal sheaves are positive along every component.

\begin{prop} \label{prop:vanishingofH1}
Let $X$ be a smooth projective variety.  Let $f: Z \to X$ be a genus $0$ stable map such that $f$ is an immersion. Suppose that $N_{f/X}|_{Z_i}$ is globally generated for every component $Z_{i}$ of $Z$. Then $f$ is a smooth point of the Kontsevich moduli space and there are deformations of $f$ which have irreducible domain and which map birationally onto free curves in $X$.
\end{prop}

Since this result is well-known we omit the proof.
We next recall some deformation theoretic properties of stable maps which will be familiar to experts. 

\begin{prop}
\label{prop:transverseintersection}
Let $X$ be a smooth variety with $D$ an arbitrary divisor on $X$ and let $C \subset X$ be a free rational curve that is general in its deformation class. Then $C$ is not tangent to $D$ at any point of intersection.
\end{prop}

\begin{proof}
The family of free curves gives a base $M$, a flat morphism $\pi: \mathcal{C} \to M$ whose general fiber is 
$\mathbb{P}^1$, and a morphism $F:  \mathcal{C} \to X$ sending a general fiber of $\pi$ birationally onto a curve in $X$. After replacing $M$ and $\mathcal{C}$ by open subsets and taking a base change over $M$, we may assume that $M$ and $\mathcal{C}$ are smooth and that $F^{-1}(D)$ with its reduced structure is the sum of a finite union $R$ of sections of $\pi$ and a $\pi$-vertical divisor $S$.  By generic smoothness, if $p$ is the intersection of $R$ with a general fiber $C$ then $p$ is a smooth point  of $F|_R: R \to D$, and therefore $T_{D,f(p)}$ is in the image of 
$T_{\mathcal{C},p} \to T_{X,f(p)}$. So $F$ is not smooth at $p$ if and only if 
the image of $C$ is tangent to $D$ at $F(p)$. On the other hand, since $f := F|_C$ is a free morphism, for any point $q$ on $C$ the tangent space to the fiber of $F$ through $q$ is $H^0(C,N_{f,X}(-q))$.  This has the expected dimension so $F$ is smooth at any point $q$ of $C$. Therefore, the image of $C$ is not tangent to $D$ at $f(p)$. 
\end{proof}

\begin{prop} \label{prop:veryfreeavoidsfibers}
Let $X$ be a smooth threefold carrying a smooth divisor $D$ which admits a fibration $\phi: D \to \bP^1$. Let $C$ be a  very free curve on $X$ that is general in its deformation class. Then $C$ meets each fiber of $\phi$ in at most one point.
\end{prop}

\begin{proof}
By \cite[II.3.14 Theorem]{Kollar} such a curve $C$ is embedded in $X$.  Let $B$ be an open set of the moduli space of curves parametrizing irreducible deformations of $C$ which are embedded in $X$. Suppose that every general deformation of $C$ meets some fiber of $\phi$ in at least two points.  Consider the family $\cU$ of tuples $(p,q,C_0)$ such that $C_{0}$ is a smooth very free curve that is a deformation of $C$ and $p,q$ are distinct points in of $C_{0}$ which both lie in the same fiber of $\phi$.  We have a map $\pi_1: \cU \to D \times D$.    This map cannot dominate since $p$ and $q$ must lie in the same fiber of $\phi$.

We claim that in this situation $\cU$ cannot dominate $B$. By Proposition \ref{prop:transverseintersection}, we see that a general curve $C$ will meet $D$ transversely at every point. Since $C$ is very free, the dimension of $B$ is $h^0(C,N_{C/X})$, while the dimension of the fiber of $\pi_1$ containing $C$ is $h^0(C,N_{C/X}(-p-q)) = h^0(C,N_{C/X}(-2)) = \dim B - 4$. Since the image of $\pi_1$ is at most 3-dimensional, it follows that $\dim \cU < \dim B$. Thus, a general curve $C$ will not meet $D$ multiple times in a fiber of $\phi$.
\end{proof}

\begin{prop} \label{prop:imageisnodal}
Let $X$ be a smooth projective variety of dimension $n$ equipped with a dominant morphism $\pi: X \to Z$ such that $\dim(Z) \geq 2$ and $\pi$ is smooth away from a codimension $2$ subset $V \subset X$.  Let $C$ be a general member of a family of very free curves on $X$.  Then $\pi|_{C}$ is birational and its image is a curve with at worst nodes in $Z$.
\end{prop}

\begin{proof}
Let $n$ denote the dimension of $X$ and let $M$ be an irreducible component of $\overline{\Rat}(X)$ whose general points correspond to very free curves.
For any $p$ in $X$ and any tangent direction at $p$, if there is an irreducible very free curve through $p$ in our family with the given tangent direction then the  locus in $M$ parametrizing such curves has codimension $2(n-1)$.  
Since $C$ is a general very free curve, we may assume that it does not intersect $V$. For any $p \in X \backslash V$ the locus of 
very free irreducible curves through $p$ whose tangent direction at $p$ is in the kernel of $T_X|_p \to \pi^*T_Z|_p$ is of codimension $ \geq 2(n-1)-( n- {\rm{rank }} \; T_{\pi,p} - 1) = n-1+ \dim Z$ in $M$. Hence the locus of curves in $X \backslash V$  whose tangent direction at a point goes to $0$  under $\pi$ has codimension $\geq \dim(Z) - 1$.
This shows that the image of a general curve parametrized by $M$ has at worst nodes. 

To show $\pi|_C$ is birational for a general $C$ parametrized by $M$, note that if $\pi|_C$ is not birational, then there should be a point at which $\pi|_C$ is not an immersion, so the result follows from the argument above. 
\end{proof}

\section{The invariants in Geometric Manin's Conjecture} \label{sec:invariants}

As mentioned in the introduction, Geometric Manin's Conjecture predicts that the behavior of rational curves on a Fano variety $X$ is controlled by two invariants, the $a$-invariant and the $b$-invariant.  
The primary input into Manin's Conjecture is the $a$-invariant.

\begin{defi}[the $a$-invariant]
Let $X$ be a smooth projective variety and let $L$ be a big and nef $\mathbb Q$-divisor on $X$. Then we define the Fujita invariant (or the $a$-invariant) by
\[
a(X, L) = \inf\{t \in \mathbb R \mid tL + K_X \in \Eff^1(X)\}.
\]
By \cite{BDPP}, $a(X, L) > 0$ if and only if $X$ is uniruled. When $L$ is nef but not big, we formally set $a(X, L) = \infty$.

When $X$ is singular, we pick a resolution $\beta: \widetilde{X} \rightarrow X$ and define the $a$-invariant by
\[
a(X, L) : = a(\widetilde{X}, \beta^*L).
\]
\cite[Proposition 2.7]{HTT15} shows that this definition does not depend on the the choice of $\beta$.
\end{defi}

Using the boundedness of singular Fano varieties (\cite{birkar16} and \cite{birkar16b}), \cite{HJ16} established the following theorem:

\begin{theo}[\cite{HJ16} Theorem 1.1 and \cite{LTCompos} Theorem 3.3] \label{theo:HJ}
Let $X$ be a smooth uniruled projective variety and let $L$ be a big and nef $\mathbb Q$-divisor on $X$. Let $V$ be the union of all subvarieties $Y$ such that $a(Y, L) > a(X, L)$. Then $V$ is a proper closed subset of $X$ and each component $V_i$ of $V$ satisfies $a(V_i, L) > a(X, L)$.
\end{theo}

The secondary input into Manin's Conjecture is the $b$-invariant.

\begin{defi}[the $b$-invariant] \label{defi:binvariant}
Let $X$ be a smooth uniruled projective variety and let $L$ be a big and nef $\mathbb{Q}$-divisor on $X$.  Let $\mathcal{F}_{X}$ denote the face of $\Nef_{1}(X)$ consisting of curve classes with vanishing intersection against $K_{X} + a(X,L)L$.   We define
\begin{equation*}
b(X,L) = \dim(\Span (\mathcal{F}_{X})).
\end{equation*}
When $L$ is nef but not big, we formally set $b(X, L) = \infty$.

When $X$ is singular, we pick a resolution $\beta: \widetilde{X} \rightarrow X$ and define the $b$-invariant by
\[
b(X, L) : = b(\widetilde{X}, \beta^*L).
\]
\cite[Proposition 2.10]{HTT15} shows that this definition does not depend on the choice of $\beta$.
\end{defi}

We will also need the following refinement of the $b$-invariant.

\begin{defi}
Let $X$ be a smooth uniruled projective variety and let $L$ be a big and nef $\mathbb{Q}$-divisor on $X$.  Suppose that $f: Y \to X$ is a generically finite dominant morphism from a smooth projective variety $Y$ satisfying $a(Y,f^{*}L) = a(X,L)$.  As in Definition \ref{defi:binvariant} let $\mathcal{F}_{X}$ denote the face of $\Nef_{1}(X)$ perpendicular to $K_{X} + a(X,L)L$ and let $\mathcal{F}_{Y}$ denote the face of $\Nef_{1}(Y)$ perpendicular to $K_{Y} + a(Y,f^{*}L)f^{*}L$.  By \cite[Lemma 4.25]{LST18} the pushforward $f_{*}$ takes $\mathcal{F}_{Y}$ into $\mathcal{F}_{X}$.  We say that $f$ is face contracting if the map $f_{*}: \mathcal{F}_{Y} \to \mathcal{F}_{X}$ is not injective.
\end{defi}

\section{Non-dominant families of rational curves}

In this section we will analyze non-dominant families of rational curves on a smooth Fano threefold $X$.  The first step is to classify the subvarieties $Y$ such that $a(Y,-K_{X}) > a(X,-K_{X})$.  Their structure is given by the following theorem:

\begin{theo} \label{theo:higherainv}
Let $X$ be a smooth Fano threefold.  Suppose that $Y$ is a $2$-dimensional subvariety of $X$ satisfying $a(Y,-K_{X}) > a(X,-K_{X})$.  Let $\psi: \widetilde{Y} \to Y$ denote a resolution of singularities and let $\phi$ denote the the composition of $\psi$ with the inclusion map to $X$.
\begin{enumerate}
\item If $\kappa(K_{\widetilde{Y}} - a(Y,-K_{X})\phi^{*}K_{X}) = 1$ then $Y$ is swept out by $-K_{X}$-lines.
\item If $\kappa(K_{\widetilde{Y}} - a(Y,-K_{X})\phi^{*}K_{X}) = 0$ then $Y$ is a contractible divisor of type E1, E2, E3, E4, or E5. Moreover if $Y$ has type E1 then $(Y, -K_X|_Y)$ is isomorphic to $(\mathbb P^1\times\mathbb P^1, \mathcal O(1,1))$ and both rulings correspond to E1 contractions. 
\end{enumerate}
\end{theo}

\begin{proof}
\textbf{Case 1:} $\kappa(K_{\widetilde{Y}} - a(Y,-K_{X})\phi^{*}K_{X}) = 1$.  Let $\rho: \widetilde{Y} \to B$ denote the Iitaka fibration for $K_{\widetilde{Y}} - a(Y,-K_{X})\phi^{*}K_{X}$.  A general fiber $C$ satisfies
\begin{equation*}
-\phi^{*}K_{X} \cdot C = \frac{-K_{\widetilde{Y}} \cdot C}{a(Y,-K_{X})} = \frac{2}{a(Y,-K_{X})}.
\end{equation*} 
Since $a(Y,-K_{X}) > a(X,-K_{X}) = 1$ we conclude that the image of $C$ in $X$ is a $-K_{X}$-line that sweeps out $Y$.

\textbf{Case 2:} $\kappa(K_{\widetilde{Y}} - a(Y,-K_{X})\phi^{*}K_{X}) = 0$.  Let $\nu: \widehat{Y} \to Y$ be the normalization map.  Since $a(Y,-K_{X}) > 1$ and the pair is adjoint rigid, \cite[Lemma 5.3]{LTJAG} shows that $\widehat{Y}$ has only canonical singularities and that $K_{\widehat{Y}} \sim_{\mathbb Q} a(Y,-K_{X})\nu^{*}K_{X}$. 
Using the classification of \cite[Proposition 1.3]{horing10}, we see that the pair $(\widehat{Y},-\nu^*K_{X})$ is isomorphic to either $(\mathbb{P}^{2},\mathcal{O}(1))$, $(\mathbb{P}^{2},\mathcal{O}(2))$, or $(Q,\mathcal{O}(1)|_{Q})$ where $Q$ is a smooth or a singular quadric hypersurface in $\mathbb{P}^{3}$. 

First suppose that $Y$ is normal so that $\nu$ is an isomorphism.  This implies that $K_Y \sim_{\mathbb Q} a(Y, -K_X)K_X|_Y$. By adjunction we conclude that $Y|_Y$ is antiample. So there is an extremal ray of $\overline{\mathrm{Eff}}_1(X)$ such that $Y \cdot R < 0$. Then this extremal ray defines a divisorial contraction contracting $Y$. We can identify the type of contraction by comparing against the classification of exceptional divisors in Theorem \ref{theo:moriclassification}.  If $(Y,-K_{X}|_{Y})$ is isomorphic to $(\mathbb{P}^{2},\mathcal{O}(1))$ the extremal contraction will have type E5.  If $(Y,-K_{X}|_{Y})$ is isomorphic to $(Q,\mathcal{O}(1))$ the extremal contraction will have type E1, E3, or E4, and furthermore in the case of an E1 contraction $Q$ must be smooth so that both rulings correspond to E1 contractions.  If $(Y,-K_{X}|_{Y})$ is isomorphic to $(\mathbb{P}^{2},\mathcal{O}(2))$ the extremal contraction will have type E2. 

It only remains to show that $Y$ cannot be non-normal.  We separate the argument into two cases depending upon whether or not $-K_{X}$ is very ample.  First suppose that $-K_X$ is very ample.  Note that the map $\nu: \widehat{Y} \to Y$ must be defined by a sublinear series of $|-\nu^{*}K_{X}|$.  If $(\widehat{Y},-\nu^*K_{X})$ is isomorphic to $(\mathbb{P}^{2},\mathcal{O}(1))$, then no strict sublinear series will define a birational map and we conclude that $Y = \widehat{Y}$ is normal.  If $(\widehat{Y},-\nu^*K_{X})$ is isomorphic to $(Q,\mathcal{O}(1))$ then again no strict sublinear series can define a birational map and we conclude that $Y = \widehat{Y}$ is normal.

Finally suppose that $(\widehat{Y},-\nu^*K_{X})$ is isomorphic to $(\mathbb{P}^{2},\mathcal{O}(2))$.  Let $V \subset |\mathcal{O}(2)|$ denote the sublinear series defining the map $\nu: \widehat{Y} \to Y$.  We may concentrate on the case when $3 \leq \dim(V) \leq 4$.  First suppose that $V$ is $4$-dimensional.  Thus $Y$ is isomorphic to a projection of the $2$-Veronese surface from a point $p \in \mathbb P^5$.  Since $\nu$ is a morphism, $p$ cannot be on $\widehat{Y}$.  Moreover, if $p$ is not on the secant variety of $\widehat{Y}$ then the projection defines an isomorphism so that $Y$ is normal.  If instead $p$ lies on the secant variety of $\widehat{Y}$, then $Y$ is a singular surface of degree $4$ which is singular along a line $\ell$. Then the preimage $C = \nu^{-1}(\ell)$ is a line on $\mathbb P^2$ and $C \to \ell$ is a degree $2$ cover. The surface $Y$ can be realized as the complete intersection of two quadrics in $\mathbb P^4$ so that $K_Y^2 = 4$ and $K_Y$ is antiample on $Y$. Thus we conclude that $K_{\widehat{Y}}+C \sim \nu^*K_Y$. By adjunction, we conclude that $Y|_Y$ is trivial. Then $|Y|$ defines a pencil $\rho : X \to \mathbb P^1$ and this must be a del Pezzo fibration of degree $4$. Let $q \in \mathbb P^1$ be the point corresponding to $Y$.  By the semicontinuity theorem $\rho_*\mathcal O(-K_X)$ is locally free of rank $5$ in a neighborhood $U$ of $q$ so that we have an embedding
\[
\rho^{-1}(U) \hookrightarrow \mathbb P(\rho_*\mathcal O(-K_X)|_U).
\]
Thus after shrinking $U$ there exist four quadrics $Q_1, Q_2, Q_3, Q_4$ such that $\rho^{-1}(U)$ is defined by
\[
f(s)Q_1(x_0 \cdots, x_4) = Q_2(x_0, \cdots, x_4), \quad g(s)Q_3(x_0 \cdots, x_4) = Q_4(x_0, \cdots, x_4)
\]
where $s =0$ corresponds to $q$ and $f, g$ are polynomials in $s$ such that $f(0) = g(0) = 0$. Note that $s = Q_2 = Q_4 =0$ is a complete intersection of two quadrics which is singular along a line $\ell$.  By arguing as in \cite[Lemma 3.1]{LTJAG}, one can prove that $\rho^{-1}(U)$ is singular. But this contradicts with the fact that $X$ is smooth.

Next we consider the case when the sublinear system $V$ is $3$-dimensional. Then $Y$ is the image of a projection of the $2$-Veronese surface from a line.
In particular $Y$ is a quartic surface.

We claim that $Y$ is isomorphic to a hyperplane section of a quartic threefold which is smooth along $Y$.  Since $K_Y$ is trivial it follows from adjunction that $Y|_Y$ is very ample. It follows from Kodaira vanishing that
\[
0 \to H^0(X, \mathcal O_{X}) \to H^0(X, \mathcal O_{X}(Y)) \to H^0(Y, \mathcal O_{Y}(Y|_Y)) \to 0
\]
is exact. We conclude that $|Y|$ is base point free and $h^0(X, \mathcal O(Y)) = 5$. Furthermore $Y^3 = 4$, hence we conclude that $|Y|$ defines a birational morphism $f : X \to X'$ to a quartic hypersurface $X'$ in $\mathbb P^4$. Note that there exists a Zariski open neighborhood $U$ of $Y \subset X'$ such that $f|_{f^{-1}(U)} : f^{-1}(U) \to U$ is bijective. Let $p \in Y \subset X'$. If this is a smooth point of $Y$, then it is a smooth point of $X'$. This means that $f$ is an isomorphism in a neighborhood of $p$. On the other hand, if $p$ is a singular point of $Y$, then the tangent space of $Y$ at $p$ is the tangent space of $X$ at $p$. Since $Y|_Y$ is very ample, we conclude that $|Y|$ is very ample in a neighborhood of $p$. Thus after shrinking $U$ we may conclude that $f|_{f^{-1}(U)} : f^{-1}(U) \to U$ is an isomorphism.

The proof of \cite[Lemma 6.14]{LTT18} shows that every hyperplane section of a quartic threefold which is smooth along that hyperplane section is normal.  Thus we conclude our assertion when $-K_X$ is very ample.

\

Next we discuss the case when $-K_X$ is not very ample. Here we use the classification of Fano threefolds whose anticanonical linear series are not very ample in Theorems~\ref{theo:notbpf} and \ref{theo:notveryample}.  For each threefold we show that there is no non-normal surface $Y \subset X$ such that $(Y,-K_{X}|_{Y})$ is adjoint rigid with $a$-invariant $> 1$.

First let us assume that $X$ is $V_1$. In this case it is proved in \cite[Proposition 6.11]{LTT18} that there is no surface with higher $a$-invariant.

Next let us assume that $X$ is the blow up of $V_1$ along a complete intersection of two members of 
$|-\frac{1}{2}K_{V_1}|$.  We denote by $H$ the pullback of the ample generator on $V_{1}$ and by $E$ the exceptional divisor.  For any divisor $j(H-E) + kE$ we have
\[
(-K_X)^2(j(H-E) + kE) = (2H-E)^2(jH + (k-j)E) = j + 2k. 
\]
If we have a divisor $Y$ such that $a(Y, -K_X) > 1$ and $(Y, -K_X|_Y)$ is adjoint rigid, then we must have $(-K_X)^2 \cdot Y = 1, 2$ or $4$. If we write $Y \sim j(H-E) + kE$, then we must have
\begin{equation*}
(j,k) \in \{ (1,0), (2,0), (4,0), (0,1), (2,1), (0,2) \}.
\end{equation*}
Since $Y$ is integral, we must have $j = 1$ and $k= 0$, or $j = 0, 2$ and $k = 1$. 
When $j = 1$ and $k = 0$, $|H -E|$ defines a del Pezzo fibration $\rho : X \to \mathbb P^1$ of degree $1$, and it follows from \cite[Lemma 6.9]{LTT18} that $Y$ is normal.  When $j = 0$ and $k = 1$, $Y$ is equal to $E$ so that $Y$ is not rational. This contradicts with the fact that $(Y, -K_X|_Y)$ is adjoint rigid. When $j = 2$ and $k = 1$, $Y \sim 2H-E\sim -K_X$. 
We show that any integral member of $|-K_{V_1}|$ is normal by mimicking \cite[Lemma 3.3]{LTJAG}.  Recall that $V_1$ is a subscheme of $\mathbb{P}(1,1,1,2, 3)$ defined by an equation
\[
f_6(x_0, x_1, x_2) + f_4(x_0, x_1, x_2)y + y^3 = z^2
\]
and a member $Y \in |-K_{V_1}|$ is defined by
\[
f_2 = cy.
\]
Let us assume that $c = 0$ (since the case $c \neq 0$ is easier).  Then we may assume that $f_2 = x_0^2 + x_1^2 + x_2^2$. Let $C$ be a $1$-dimensional component of the singular locus of $Y$.  Since the Jacobian of the two equations defining $Y$ has rank $1$ along $C$ we obtain a map $C \to \mathbb{P}(1, 5)$ recording the linear relation between the two rows.  If this map is surjective, then the gradient of the equation defining $V_{1}$ will vanish identically at some point of $C$, contradicting the smoothness of $V_{1}$.  If it is not surjective, then there is some constant $a$ such that along $C$ we have $ax_i^5 = \partial (f_6 + f_4y)/\partial x_i$ for every $i$.  
O the other hand we must have $2x_i  \partial (f_6 + f_4y)/\partial x_j = 2x_j  \partial (f_6 + f_4y)/\partial x_i$ for any $i, j$ which implies that $2ax_i x_j^5 = 2ax_jx_i^5$ along $C$. This is impossible unless $a = 0$. But this means that $X$ is singular, a contradiction.

Next let us assume that $X$ the blow up of $V_2$.  This can be proved just like the case of the blow up of $V_1$.

Next let us discuss the case $X = \mathbb P^1 \times S_1$ where $S_1$ is a degree $1$ del Pezzo surface. We denote the projections by $p_1 : X \to \mathbb P_1$ and $p_2 : X \to S_1$. We denote the pullback of the hyperplane class from $\mathbb P^1$ by $h$. Let $Y \subset X$ be a divisor such that $a(Y, -K_X) >1$ and $(Y, -K_X|_Y)$ is adjoint rigid. Then there exists an effective class $\alpha$ on $S_1$ such that $Y \sim ah + p_2^*\alpha$ with $a \geq 0$. Then we have
\[
(-K_X)^2 \cdot Y = a + (-4K_{S_{1}} \cdot \alpha).
\]
Since we must have $(-K_X)^2 \cdot Y = 1, 2$, or $4$ and $Y$ is integral, we conclude that either $a = 1$ and $\alpha = 0$ or $a = 0$ and $-K_{S_{1}} \cdot \alpha = 1$.  The former case $Y$ is a fiber of $p_1$ and thus normal.  In the latter case $(Y, -K_X|_Y)$ is not adjoint rigid. Thus we conclude our assertion.

Next let us assume that $X = \mathbb P^1 \times S_2$ where $S_2$ is a degree $2$ del Pezzo surface.  This can be proved just like the case of $\mathbb P^1 \times S_1$.

Next let us assume that $X$ is the double cover of $\mathbb P^3$ ramified along a smooth sextic.
Suppose $Y$ is a divisor such that $a(Y, -K_X) >1$ and $(Y, -K_X|_Y)$ is adjoint rigid. Since $(-K_{X})|_{Y}^{2} \leq 4$ we have either $Y \in |-K_X|$ or $|-2K_X|$. In the former case, \cite[Lemma 6.6]{LTT18} shows that $Y$ is normal. In the latter case, one may prove that $Y$ is normal by mimicking the argument of \cite[Lemma 3.3]{LTJAG} (as in the case of $V_{1}$). 

Next let us assume that $X$ is the double cover of a quadric hypersurface.
Suppose $Y$ is a divisor such that $a(Y, -K_X) >1$ and $(Y, -K_X|_Y)$ is adjoint rigid.  Since $(-K_{X})|_{Y}^{2} \leq 4$  we must have $Y \in |-K_X|$. However, arguing as in \cite[Lemma 3.3]{LTJAG} (as in the case of $V_{1}$) one can prove that $Y$ is normal.  

Finally let us assume that $X$ is the double cover of $\mathbb P^1 \times \mathbb P^2$ branched over a smooth hypersurface of degree $(2,4)$.  Let $H_1$ be the pullback of the hyperplane class from $\mathbb P^1$ and $H_2$ be the pullback of the hyperplane class from $\mathbb P^2$. Then we have $-K_X \sim H_1 + H_2$. Suppose $Y$ is a divisor such that $a(Y, -K_X) > 1$ and $(Y, -K_X|_Y)$ is adjoint rigid.  One can write $Y \sim jH_1 + kH_2$ where $j, k$ are non-negative integers. Then we have
\[
(-K_X)^2 \cdot Y = 2j + 4k.
\]
Since $Y$ is integral, we conclude that either $j = 1$ and $k = 0$ or $j = 0$ and $k = 1$.
When $j = 1$ and $k = 0$, $Y$ is a fiber of a degree $2$ del Pezzo fibration $\pi_1 : X \to \mathbb P^1$. Let us show that every fiber of $\pi_1$ is normal. We introduce coordinates $(s:t)$ for $\mathbb P^1$ and $(x:y:z)$ for $\mathbb P^2$. Then there exist degree $4$ homogenous polynomials $f, g, h$ in $x, y, z$ such that $X$ is a GIT quotient of a hypersurface defined by
\[
\begin{pmatrix}s & t \end{pmatrix} \begin{pmatrix}f & h \\ h & g \end{pmatrix}\begin{pmatrix}s \\ t \end{pmatrix} = w^2,
\]
where the group action of $\mathbb G_m^2$ is given by
\[
(t_1, t_2) \cdot (s,t, x,y,z w) \mapsto (t_1s,t_1t, t_2x,t_2y,t_2z, t_1t_2^2w).
\]
Assume that the fiber corresponding to $(s:t) = (1:0)$ is singular in codimension $1$ and let $C$ denote a $1$-dimensional irreducible component of the singular locus.  Since such a curve must lie in the ramification locus of the double cover mapping this fiber to $\mathbb{P}^{2}$, we see that the equation $f$ defining the branch divisor is non-reduced. The gradient of the above equation along $Y$ is
\[
(2sf, 2sh, s^2\partial f/\partial x, s^2\partial f/\partial y, s^2\partial f/\partial z, 2w).
\]
Along $C$, we have $f = \partial f/\partial x = \partial f/\partial y = \partial f/\partial  z = w =0$. Then $C$ meets with the locus defined by $h = 0$, but this contradicts with the fact that $X$ is smooth. Thus we conclude that $Y$ is normal.

When $j = 0$ and $k = 1$, $Y$ is the pullback of a line $\ell$ on $\mathbb P^2$ via the conic bundle $\pi_2 : X \to \mathbb P^2$. If $\pi_2$ admits a non-reduced conic, then one can prove that $X$ must be singular. Thus we conclude that any fiber of $\pi_2$ is either a smooth conic or the union of two distinct lines. This implies that if the discriminant locus of $\pi_2$ contains a line, then the pullback of that line is reducible. However, this contradicts with the fact that all effective divisors on $X$ are nef. Thus we conclude that the discriminant locus of $\pi_2$ does not contain a line. Thus it follows from \cite[Lemma 7.2]{LTT18} that $Y$ is normal. 
Thus our assertion follows.
\end{proof}

In order to translate this result to a theorem about rational curves, we will need a result from \cite{LTCompos}.

\begin{theo}[\cite{LTCompos} Theorem 1.1] \label{theo:ainvanddeformations}
Let $X$ be a smooth weak Fano variety. As in Theorem \ref{theo:HJ} let $V$ denote the union of all subvarieties $Y$ such that $a(Y, -K_X) > a(X, -K_X)$. Then any component $M$ of $\overline{\Rat}(X)$ parametrizing a non-dominant family of rational curves will parametrize rational curves in $V$.
\end{theo}

By combining Theorem \ref{theo:higherainv} and Theorem \ref{theo:ainvanddeformations} we obtain Theorem \ref{theo:maintheorem1}.

\subsection{Low degree curves with higher than expected dimension}
We will also need some more precise information about low degree curves which deform more than expected.  We will focus on $-K_{X}$-lines and $-K_{X}$-conics, i.e.~irreducible rational curves of anticanonical degree $1$ or $2$.

\begin{lemm}
\label{lemm:lines}
Let $X$ be a smooth Fano threefold.  Suppose that there is a family of $-K_X$-lines which has dimension $>1$.  Then these lines sweep out a contractible divisor $Y$ of E5 type.
\end{lemm}

If $\ell$ is a $-K_{X}$-line whose parameter space has dimension $>1$ we will say that $\ell$ is a line of E5 type.  This result was established in \cite[Ch. III Proposition 2.1]{Isk79} for non-hyperelliptic Fano threefolds; we give a general proof using $a$-invariants.

\begin{proof}
Let $M$ be a component of the Hilbert scheme parametrizing $-K_X$-lines and let $d = \dim M$.  Since any $-K_X$-line $\ell$ satisfies $a(\ell, -K_X) = 2 > 1 = a(X, -K_X)$, by Theorem \ref{theo:HJ} the lines parametrized by $M$ sweep out a surface $Y$.  By applying Theorem \ref{theo:HJ} to a resolution of $Y$ we see that $a(Y, -K_X) \geq a(\ell, -K_{X} ) = 2$.  By \cite[Proposition 1.3]{horing10} we have only two possibilities: either $a(Y, -K_X) = 2$ or $3$. 

First suppose that $a(Y, -K_X) = 2$.  Let $\phi : \widetilde{Y} \to Y$ be a resolution and let $\widetilde{\ell}$ be the strict transform of a general line $\ell$. Then we have
\[
-K_{\widetilde{Y}} \cdot \widetilde{\ell} - 1 = d.
\] 
On the other hand since $\widetilde{\ell}$ is a nef class on $\widetilde{Y}$ we must have
\[
0 \leq (-2\phi^*K_X + K_{\widetilde{Y}}) \cdot \widetilde{\ell} = 1-d.
\]
Thus we conclude that $d = 1$, showing that in this case $M$ cannot have larger than the expected dimension.

Thus under our assumptions $a(Y, -K_X) = 3$.  \cite[Proposition 1.3]{horing10} shows that $(Y,-K_{X}|_{Y})$ is birationally equivalent to $(\mathbb{P}^{2},\mathcal{O}(1))$, and in particular must be adjoint rigid.  Theorem \ref{theo:higherainv} shows that $Y$ is a contractible divisor and it must have E5 type.
\end{proof}

We will need to strengthen this result to include non-rational curves:

\begin{lemm} \label{lemm:linesarerational}
Let $X$ be a smooth Fano threefold.  Suppose that there is a family of curves $C$ of anticanonical degree $1$ and dimension $\geq 2$ which sweeps out a surface $Y$.  Then the curves are rational and $Y$ is a contractible divisor of E5 type.
\end{lemm}

\begin{proof}
By Lemma \ref{lemm:lines} it suffices to show that any such family of curves will parametrize rational curves.  

If $|-K_{X}|$ is basepoint free, then the image of $C$ under the corresponding morphism will be a line in projective space, hence rational.  Furthermore, $C$ will be taken birationally onto its image under this map since $-K_{X} \cdot C = 1$.  Thus $C$ will be rational.

We only need to address the Fano threefolds for which $|-K_{X}|$ is not basepoint free.  The three types are listed in Theorem \ref{theo:notbpf}:
\begin{enumerate}
\item If $X = V_{1}$ then the anticanonical divisor has index $2$ and thus there are no curves of anticanonical degree $1$.
\item Suppose $X$ is the blow up of $V_{1}$ along a smooth elliptic curve which is the intersection of two members of $|-\frac{1}{2}K_{V_{1}}|$.  If $\phi$ denotes the blow-up and $E$ denotes the exceptional divisor then we can write $-K_{X} = -\phi^{*}\frac{1}{2}K_{V_{1}} + (-\phi^{*}\frac{1}{2}K_{V_{1}} - E)$ where both terms are nef.  Thus the curves of anticanonical degree $1$ admit the following descriptions:
\begin{enumerate}
\item Curves with $-\phi^{*}\frac{1}{2}K_{V_{1}} \cdot C = 1$ and $(-\phi^{*}\frac{1}{2}K_{V_{1}} - E) \cdot C = 0$.  These curves are contracted by the morphism $g: X \to \mathbb{P}^{1}$ whose general fiber is a del Pezzo surface of degree $1$.  By \cite[Lemma 6.9]{LTT18} every fiber of $g$ is irreducible and normal and is thus a Gorenstein del Pezzo surface.  The fibers which are not smooth either have canonical singularities or are isomorphic to a cone over an elliptic curve.  In the latter case, the only curves of anticanonical degree $1$ in the fiber are the lines in the cone which are rational. For all other fibers, any curve of anticanonical degree $1$ will either be a $(-1)$-curve or an element of $|-K_{F}|$.  Curves of the first type will be rational.  Curves of the second type will form a dominant $2$-dimensional family on $X$.  In particular, there is no subfamily of dimension $>1$ which sweeps out a surface $Y$.
\item Curves with $-\phi^{*}\frac{1}{2}K_{V_{1}} \cdot C = 0$ and $(-\phi^{*}\frac{1}{2}K_{V_{1}} - E) \cdot C = 1$.  These are the rational curves contracted by the birational map to $V_{1}$.
\end{enumerate}
\item If $X = \mathbb{P}^{1} \times S_{1}$, then the curves of anticanonical degree $1$ are the $(-1)$-curves and the curves in the anticanonical linear series in some fiber of the map to $\mathbb{P}^{1}$.  Curves of the first type are rational.  The total parameter space for curves of the second type has dimension $2$ and this family sweeps out all of $X$.  Thus there is no subfamily of dimension $>1$ which sweeps out a surface $Y$.
\end{enumerate}
\end{proof}

We will also need a statement concerning the geometry of families of $-K_{X}$-lines of dimension $1$.

\begin{lemm} \label{lemm:linesfromruled}
Let $X$ be a smooth Fano threefold.  Suppose that $M$ is a family of $-K_{X}$-lines on $X$ which has dimension $1$.  There is a ruled surface $\pi: S \to C$ equipped with a morphism $f: S \to X$ such that the fibers of $\pi$ map birationally to the lines parametrized by $M$.
\end{lemm}

\begin{proof}
Let $C$ be the normalization of the reduced curve underlying $M$ and let $S$ be a minimal resolution of the base-change of the one-pointed family over $M$ to $C$.  Note that $S$ is equipped with an evaluation map $f: S \to X$ that maps each fiber of $\pi: S \to C$ onto a line parametrized by $M$.  More precisely, for every fiber $F$ of $\pi$ there is a unique component $F_{0}$ of multiplicity $1$ in $F$ which is mapped birationally onto a line parametrized by $M$ and every other component of $F$ is contracted by $f$.  Since $S$ is minimal, we may assume that no $(-1)$-curve on $S$ is contracted by $f$.

Suppose that the map $\pi: S \to C$ has a reducible fiber $F$.  Since all but one component of $F$ is contracted by $f$ we see that the unique component $F_{0}$ of $F$ must be a $(-1)$-curve and this must be the only $(-1)$-curve in $F$.  But by \cite[Lemma 4.3]{LT19} $F_{0}$ must have multiplicity $\geq 2$ in $F$.  This contradicts the construction above.  We conclude that every fiber of $\pi$ is irreducible and has multiplicity $1$ so that $\pi: S \to C$ is a $\mathbb{P}^{1}$-bundle.
\end{proof}

The situation for conics is very similar.

\begin{lemm}
\label{lemm:conics}
Let $X$ be a smooth Fano threefold.  Suppose that there is a family of $-K_X$-conics which has dimension $> 2$. Then these conics sweep out a contractible surface $Y$ which has E1, E3, E4, or E5 type. Moreover if $Y$ has E1 type then $(Y, -K_X|_Y)$ is isomorphic to $(\mathbb P^1\times\mathbb P^1, \mathcal O(1,1))$ and both rulings correspond to E1 contractions. 
\end{lemm}

Thus if $C$ is a conic whose parameter space has dimension $>2$ we will say that $C$ is a conic of E1, E3, E4, or E5 type depending on the type of the surface swept out by deformations of $C$.  This result was established in many cases by \cite[Ch. III Proposition 1.3 and Proposition 3.3]{Isk79}; we give a general proof using the $a$-invariant.

\begin{proof}
Let $M$ be a component of the Hilbert scheme parametrizing these conics and let $d = \dim M$.  Since $M$ has larger than the expected dimension, the conics parametrized by $M$ sweep out a surface $Y$.  By Theorem \ref{theo:ainvanddeformations} we have $a(Y, -K_X) > 1$. By \cite[Proposition 1.3]{horing10} we have only three possibilities: $a(Y, -K_X) = 3/2$, $2$, or $3$.

First suppose that $a(Y,-K_{X}) = 3/2$.  Let $\phi : \widetilde{Y} \to Y$ be a resolution and $\widetilde{C}$ be the strict transform of a general deformation of $C$.  Since the family of deformations of $\widetilde{C}$ has the expected dimension on $\widetilde{Y}$ we must have
\[
-K_{\widetilde{Y}} \cdot \widetilde{C} - 1 = d
\]
and $d \geq 3$.  On the other hand since $\widetilde{C}$ is nef on $\widetilde{Y}$ we must have
\[
0 \leq \left(\frac{-3}{2}\phi ^*K_X + K_{\widetilde{Y}} \right) \cdot \widetilde{C} = 3-1-d.
\]
This gives a contradiction, showing that this case cannot happen.

Second suppose that $a(Y, -K_X) = 2$.  Let $\phi  : \widetilde{Y} \to Y$ be a resolution and $\widetilde{C}$ be the strict transform of a general deformation of $C$. Then we must have
\[
-K_{\widetilde{Y}} \cdot \widetilde{C} - 1 = d
\]
and $d \geq 3$.  On the other hand since $\widetilde{C}$ is nef on $\widetilde{Y}$ we must have
\[
0 \leq (-2\phi ^*K_X + K_{\widetilde{Y}}) \cdot \widetilde{C} = 4-1-d.
\]
Thus in this case $d=3$.  Suppose that $(Y,-K_{X})$ is not adjoint rigid.  Since we have verified that $\widetilde{C}$ has vanishing intersection against $K_{\widetilde{Y}} -2\phi ^{*}K_{X}$ it must be a general fiber for the canonical fibration for $K_{\widetilde{Y}} -2\phi ^*K_X$.  However such a fiber $F$ must satisfy $-\phi ^*K_X \cdot F = 1$, yielding a contradiction.

Thus if $a(Y, -K_X) = 2$ then $(Y,-K_{X})$ must be adjoint rigid.  By \cite[Proposition 1.3]{horing10} this pair must be birationally equivalent to $(Q,\mathcal{O}(1)|_{Q})$ where $Q$ is a (possibly singular) quadric in $\mathbb{P}^{3}$.  By Theorem \ref{theo:higherainv} we see that $Y$ is contractible, and thus must have E1, E3 or E4 type.  

Finally suppose that $a(Y,-K_{X}) = 3$.  Arguing just as in Lemma \ref{lemm:lines} we see that $Y$ is an E5 divisor.  In this case $M$ parametrizes the family of conics on $Y$.
\end{proof}

Although Lemma \ref{lemm:lines} and Lemma \ref{lemm:conics} show that any family of $-K_{X}$-lines or conics that has higher-than-expected dimension sweeps out a contractible divisor, the following example shows that the analogous statement is not true for higher degree rational curves.  (Rather, the correct analogue is that a family of degree $d$ rational curves that moves in dimension $\geq 2d-1$ will sweep out one of the exceptional divisors in Lemma \ref{lemm:conics}.  The argument is essentially the same as the proof of Lemma \ref{lemm:conics}.)

\begin{exam}
\cite[Proposition 5.4.4]{KPS18} gives several examples of Fano threefolds of Picard rank 1 and genus 12 for which the Hilbert scheme of lines is a union of rational curves. Considering the universal family of lines over the desingularization of one of the irreducible components of the Hilbert scheme, 
we get a Hirzebruch surface $\mathbb{F}_r$ and a universal map $f: \mathbb{F}_r \to X$ such that $f$ sends the fibers of $\pi: \mathbb{F}_r \to \mathbb{P}^1$  isomorphically onto lines in $X$. Let $F$ denote the class of a fiber of $\pi$ and $C$ denote the class of the section with negative self intersection. Then $f^*(-K_X) = C+ mF$ for some $m \geq r$. Pick any $b  \geq m$. There are sections of $\pi$ whose class is $C+bF$, so the moduli space of rational curves in $\mathbb{F}_r$ with class $C+bF$  is non-empty and has therefore the expected dimension
\begin{align*}
-K_{\mathbb{F}_r} \cdot (C+bF) - 1 & = (2C+(2+r)F) \cdot (C+bF) -1 \\
& =  -2r+2b+2+r-1 \\
& =2b+1-r
\end{align*}
as a family of curves of $\mathbb{F}_{r}$. But the anticanonical degree of the image of these curves in $X$ is 
$-f^*K_X \cdot (C+bF)= -r+b+m< 2b+1-r$. The image of $\mathbb{F}_r$ in $X$ is a surface $S$ which is not contractible since $X$ is of Picard rank 1, but $S$ is swept out 
by a family  of rational curves of degree $b+m - r$ which has higher-than-expected dimension. 
\end{exam}

\section{$a$-covers} \label{sect:acovers}

\begin{defi}
Let $X$ be a smooth projective uniruled variety and let $L$ be a big and nef $\mathbb{Q}$-divisor on $X$.  An $a$-cover of $(X,L)$ is a generically finite dominant morphism $f: Y \to X$ from a projective variety $Y$ satisfying $a(Y,f^{*}L) = a(X,L)$.
\end{defi}

Definition \ref{defi:manincomponent} shows that $a$-covers play a key role in the description of the exceptional set in Geometric Manin's Conjecture.  In this section our goal is to completely classify the $a$-covers of Fano threefolds.  Note that $a(X,-K_{X}) = 1$, so any $a$-cover will have $a$-invariant $1$ as well.  We will separate the argument into cases based on the Iitaka dimension of $K_{Y} - f^{*}K_{X}$.

\subsection{Iitaka dimension 2}

\begin{lemm} \label{lemm:iitakadim2case}
Let $X$ be a smooth Fano threefold.  Let $f : Y \rightarrow X$ be an $a$-cover such that $Y$ is smooth and $\kappa(K_{Y}-f^{*}K_{X}) = 2$.
After applying a resolution we may assume that the Iitaka fibration for $K_{Y}-f^*K_X$ is a morphism $\pi: Y \rightarrow Z$.  Then a general fiber of $\pi$ maps birationally under $f$ to a $-K_{X}$-conic.

If $p: \mathcal{C} \to M$ denotes the one-pointed family over the component of $\overline{\Rat}(X)$ parametrizing these conics on $X$, then $\pi$ is birationally equivalent to a base change of $p$ over a rational map $Z \dashrightarrow M$ and $f$ factors rationally through the evaluation map for $\mathcal{C}$.  
\end{lemm}

\begin{proof}
By running the $(K_{Y} - f^{*}K_{X})$-MMP we obtain a birational contraction $\phi: Y \dashrightarrow Y'$ and a morphism $\pi': Y' \to Z$ where $Z$ has dimension $2$.  Furthermore we know that $K_{Y'} - \phi_{*}f^{*}K_{X}$ is numerically trivial along a general fiber of $\pi'$.  As in the statement we replace $Y$ by a birational model which resolves the rational map to $Z$.  Since a general fiber $C'$ of $\pi'$ is a $-K_{Y'}$-conic that avoids the locus where $\phi^{-1}$ is not a morphism, we conclude that the strict transform $C$ of such a curve on $Y$ is a $-K_{Y}$-conic contracted by $\pi$.  Furthermore, we have
\begin{equation*}
(K_{Y} - f^{*}K_{X}) \cdot C = (K_{Y'} - \phi_{*}f^{*}K_{X}) \cdot C' = 0
\end{equation*}
so that $f_{*}C$ has anticanonical degree $2$.  Since any dominant family of curves has anticanonical degree at least $2$, $f_{*}C$ cannot be a multiple curve, and we deduce that $f$ maps $C$ birationally onto its image.

Since the morphism $\pi$ yields a family of stable maps over an open subset of $Z$, we obtain a dominant rational map $Z \dashrightarrow M$.  Since a dense open subset of $M$ can be embedded in the Hilbert scheme (and thus satisfies a universal property), the restriction of $\pi$ to a non-empty open subset of $Z$ is obtained by base change from the universal family over an open subset of $M$.
\end{proof}

\subsection{Iitaka dimension 1}

Let $X$ be a smooth Fano threefold.  Suppose that $X$ admits a morphism $\pi: X \to B$ with connected fibers to some curve $B$.  Then the general fiber of $\pi$ is a del Pezzo surface.  If we choose a finite map $T \to B$ and set $Y = X \times_{B} T$, then the induced $f: Y \to X$ is an $a$-cover such that $\kappa(K_{Y}-f^{*}K_{X}) = 1$.  Conversely, we will show that when $-K_{X}$ is very ample every $a$-cover $f: Y \to X$ such that $\kappa(K_{Y}-f^{*}K_{X}) = 1$ comes from this construction. 

\begin{theo} \label{theo:iitakadim1case}
Let $X$ be a smooth Fano threefold such that $-K_X$ is very ample. Let $f : Y \rightarrow X$ be an $a$-cover such that $Y$ is smooth and $\kappa(K_{Y}-f^{*}K_{X}) = 1$.
After applying a resolution we may assume that the Iitaka fibration for $K_Y-f^*K_X$ is a morphism $\pi: Y \rightarrow Z$. Let $Y_z$ be a general fiber of $\pi$ and $S_z$ be the image of $Y_z$ on $X$. Then $S_{z}$ is normal and  the linear system $|S_z|$ defines a del Pezzo fibration $\rho : X \rightarrow B$ and $\pi$ is birational to a base-change of $\rho$.
\end{theo}

\begin{proof}
Our assumption implies that $a(Y_z, -f^*K_X)= 1$ and $(Y_z, -f^*K_X)$ is adjoint rigid.  By \cite[Lemma 4.9]{LST18} this implies that $a(S_z, -K_X) = 1$ and that $(S_z, -K_X)$ is adjoint rigid (after applying a resolution).

First suppose that $S_z$ is normal.  By \cite[Lemma 5.3]{LTJAG} we see that $S_{z}$ has only canonical singularities and that $K_{S_{z}} \sim \nu^*K_{X}$. Then by adjunction we conclude that $S_z|_{S_z} \sim 0$ so that two general $S_{z}$ are disjoint.  Note that the $S_{z}$ are algebraically equivalent on $X$, and hence linearly equivalent.  Since two general $S_{z}$ are disjoint, \cite[Theorem 2.1]{Totaro00} guarantees that they are the fibers of a map $\rho: X \to B$ to a curve $B$.  It is then clear that the general fiber of $\rho$ must be a smooth del Pezzo surface.

In this case we still must show that $\pi$ is birational to a base-change of $\rho$.  \cite[Theorem 6.2]{LT16} shows that a smooth del Pezzo surface does not admit any adjoint rigid $a$-cover of degree $\geq 2$ so that $f|_{Y_z} : Y_z \rightarrow S_z$ must be birational.  Now consider the composition $\rho \circ f: Y \to B$.  The Stein factorization of this map must be $Z$, so that we obtain a finite map $h: Z \to B$.  Thus we get an induced map $g: Y \to X \times_{B} Z$.  Using the birationality of $f|_{Y_{z}}$, we see that $g$ is also birational.

It only remains to show that when $-K_{X}$ is very ample then a general $S_z$ must be normal.  Let $\nu: \widehat{S}_{z} \to S_{z}$ denote the normalization map.   Again applying \cite[Lemma 5.3]{LTJAG} we see that $\widehat{S}_{z}$ has only canonical singularities and that $K_{\widehat{S}_{z}} \sim \nu^*K_{X}$.  Since $\widehat{S}_z$ is a del Pezzo surface with canonical singularities, we conclude that $(-K_{X})^{2} \cdot S_{z} = (-K_{\widehat{S}_z})^2 \leq 9$.

Since $S_{z}$ is movable it is nef, as the two properties are equivalent for Fano threefolds.  In fact, we claim that a stronger property is true: assume that $(-K_X)^2 \cdot S_z \geq 5$. Then there is no curve $C \subset X$ such that a general $S_{z}$ will contain $C$.  
We apply a relative MMP over $Z$ with respect to the divisor $K_{Y} - f^{*}K_{X}$.  The result will be a fibration $\widetilde{Y} \to Z$ whose general fiber is a smooth weak del Pezzo surface.  Note that for a general fiber of $\pi$ this MMP will be contracting $(-1)$-curves which must satisfy $-f^{*}K_{X} \cdot C = 0$.  We conclude that over some open subset $Z^{\circ}$ every curve contracted by this relative MMP will also be contracted by the morphism $f: Y \to X$.  Thus the preimage $\widetilde{Y}^{\circ}$ still admits a morphism $\widetilde{f}: \widetilde{Y}^{\circ} \to X$.  We next replace $\widetilde{Y}^{\circ} \to Z^{\circ}$ by its relative anticanonical model $\widehat{Y}^{\circ}$.  The curves contracted by this operation satisfy $-f^{*}K_{X} \cdot C = K_{\widetilde{Y}} \cdot C = 0$, and so again arguing as above we get a map $\widehat{f}: \widehat{Y}^{\circ} \to X$ and a map $\widehat{\pi}: \widehat{Y}^{\circ} \to Z$. Since we assumed $(-K_X)^2 \cdot S_z \geq 5$, we can apply \cite[Lemma 5.3]{LTJAG} to conclude that the restriction of $\widehat{f}$ to a general fiber $\widehat{Y}_{z}$ will be the normalization map for $S_{z}$.  Since $K_{\widehat{Y}_{z}} \sim \widehat{f}^*K_{X}|_{\widehat{Y}_{z}}$ the ramification divisor for $\widehat{f}$ will not meet a general fiber $\widehat{Y}_{z}$.  This means that there is no horizontal divisor contracted by $\widehat{f}$, proving our claim.  Moreover when $(-K_X)^2 \cdot S_z = 9$, we have $\widehat{S}_z\cong \mathbb P^2$. 

By the classification of Fano threefolds with $-K_{X}$ very ample and by the classification of their nef cones in \cite{Mat95}, \cite{Fuj16}, it is rare for a Fano threefold with $-K_{X}$ very ample to carry a nef divisor $D$ such that $(-K_{X})^{2} \cdot D \leq 9$.  We will consider the different possibilities one-by-one and conclude in each case that there cannot be a one-dimensional family of non-normal surfaces $S_{z}$ as above.  There are two general cases and also several exceptional cases where $X$ has low degree and Picard rank.  We first consider the two general cases:
\begin{itemize}
\item \textbf{Case 1:} $D$ is contracted by a del Pezzo fibration.  Then by generic smoothness we see that a general deformation of $D$ will be normal.
\item \textbf{Case 2:} $D$ is contracted to a smooth curve by a conic fibration.  Almost always the base of the conic fibration will be $\mathbb{P}^{2}$, in which case we can conclude by \cite[Lemma 7.2]{LTT18} that there is no one-parameter family of non-normal surfaces in this linear series.  In the few cases when the base is not isomorphic to $\mathbb{P}^{2}$, the argument of \cite[Lemma 7.2]{LTT18} will still work.
\end{itemize}
There are also several additional possibilities when $X$ has low degree and Picard rank:

\begin{itemize}
\item \textbf{Picard rank 1:} For low degree Fano threefolds of Picard rank $1$ with $-K_{X}$ very ample, one needs to analyze the normality of elements in $|-K_{X}|$ and (for quartic hypersurfaces) in $|-2K_{X}|$.  In the first case normality is proved by \cite[Lemma 3.1]{LTJAG}, and in the second case \cite[Lemma 5.1]{LTJAG} proves that there cannot be any family of non-normal surfaces $S_{z}$ as above.
\item \textbf{Primitive Picard rank 2:} This situation is analyzed in \cite[Section 7]{LTT18}.  In every case but one the statements in  \cite[Section 7]{LTT18} show that there cannot be a family of non-normal surfaces $S_{z}$ as above.  The last case is when $X$ is a double cover of $\mathbb{P}^{2} \times \mathbb{P}^{1}$ ramified over a $(2,2)$-hypersurface.  But then the argument of \cite[Section 7.6]{LTT18} shows that any family of surfaces $S_{z}$ as above must be contracted by a del Pezzo or conic fibration, so we reduce to a previous case.
\item \textbf{Non-primitive Picard rank 2:} There are four additional cases. The first is when $X$ is the blow-up of $\mathbb{P}^{3}$ along the intersection $Z$ of two cubics and the surface $D$ is the strict transform of a hyperplane $H$ in $\mathbb{P}^{3}$.  Recall that since $(-K_X)^2 \cdot D = 7$, we showed earlier that the family of surfaces $S_{z}$ cannot contain a fixed curve in common.  Thus it suffices to consider those hyperplanes $H$ which does not contain a flex line of $Z$.
The only way that the strict transform of $H$ can be singular is when $H$ is tangent to $Z$ at a point of intersection.  If $H$ is simply tangent to $Z$, then at the blown-up point $D$ will be formal-locally isomorphic to the blow up of $\mathbb{A}^{2}$ along the ideal $(x^{2},y)$.  If $H$ meets $Z$ with multiplicity $3$, then at the blown-up point $D$ will be formal-locally isomorphic to the blow up of $\mathbb{A}^{2}$ along the ideal $(x^{3},y)$.  But in both cases $D$ will be normal.   
We conclude that there cannot be a one-parameter family of non-normal surfaces $D$ as above which share no curve in common.

The second case is when $X$ is a blow-up of a cubic hypersurface along a plane cubic $Z$ and $D$ is the pullback of a hyperplane section not containing the plane cubic. 
In this case we have $(-K_X)^2 \cdot D =9$, so if we have a one-parameter family of non-normal surfaces $D$ satisfying the conditions above then their normalizations should be isomorphic to $\mathbb P^2$. But this contradicts with the fact that $D$ is the blow up of a hyperplane section so that its normalization has Picard rank at least $2$.

The third case is the double cover of $V_7 \cong \mathbb P_{\mathbb P^2}(\mathcal O \oplus \mathcal O(1))$ branched over $B \in |-K_{V_7}|$ and $D$ is the pullback of a hyperplane from $\mathbb P^3$ such that the hyperplane does not contain the blow up point of $V_7 \to \mathbb P^3$. In this case $B$ is smooth, so one can show that $B\cap D'$ is reduced where $D'$ is the pullback of a hyperplane not containing the blow up point. Thus $D$ is smooth in codimension $1$, proving the normality of $D$.

The last case is the blow up of $\mathbb P^3$ along a curve $Z$ of degree $7$ and genus $5$ which is the intersection of three cubics and $D$ is the pullback of a hyperplane from $\mathbb P^3$. Since we have $(-K_X)^2 \cdot D = 9$, then the normalization of $D$ is isomorphic to $\mathbb P^2$. But this contradicts with the fact that $D$ has Picard rank at least $2$. 

\end{itemize}

This exhausts all the possibilities.  We conclude that when $-K_{X}$ is very ample then a general $S_{z}$ must be normal, completing the proof.
\end{proof}

In an earlier version of the current paper, we incorrectly stated Theorem \ref{theo:iitakadim1case} without the assumption that $-K_X$ is very ample.  However the statement can fail if we do not assume this property.  The following example was pointed out to us by Eric Jovinelly.

\begin{rema}
\label{rema:counterexample}
Suppose $X \cong \mathbb{P}^{1} \times S$ where $S$ is a del Pezzo surface of degree $2$.  The family of rational curves in $|-K_{S}|$ defines a finite morphism $g: W \to S$ such that $a(W,-g^{*}K_{S}) = 1$ where $W$ is a $\mathbb{P}^{1}$-fibration over a quartic curve.  If we set $Y = W \times \mathbb{P}^{1}$, then the induced morphism $f: Y \to X$ will have the property that $a(Y,-f^{*}K_{X}) = 1$ and the Iitaka fibration $\pi: Y \to T$ will be a quadric surface fibration over a quartic curve.  However, there is no del Pezzo fibration on $X$ which contracts the images $S_{z}$ of the fibers of $\pi$.
\end{rema}

\subsection{Iitaka dimension 0}

Finally, we consider the case of $a$-covers which have Iitaka dimension $0$.

\begin{theo} \label{theo:iitakadim0case}
Let $X$ be a smooth Fano threefold.  There are no $a$-covers $f: Y \to X$ such that $Y$ is smooth and $\kappa(K_{Y}-f^{*}K_{X}) = 0$.
\end{theo}

In \cite{Sen17} Sengupta shows that if $X$ is a Fano variety then for any $a$-cover $f: Y \to X$ such that $\kappa(K_{Y}-f^{*}K_{X}) = 0$ the components of the branch divisor of $f$ will have larger $a$-invariant than $X$.  We will use explicit threefold birational geometry from \cite{Kawakita05} to prove a stronger restriction on the geometry of such branch divisors.

Suppose that $Y$ and $W$ are threefolds with $\mathbb{Q}$-factorial terminal singularities and $\pi: Y \to W$ is a divisorial contraction taking the exceptional divisor $E$ to a point $P$.   Since we are only interested in the local behavior near $P$, we will henceforth assume that $W$ is a germ near $P$.  We can write
\begin{equation*}
K_{Y} \sim f^{*}K_{W} + \frac{a}{n}E
\end{equation*}
where $n$ is the index of the singularity.  In \cite{Kawakita05} Kawakita separates divisorial contractions into two types: the exceptional cases and the ordinary cases.  A key tool is the constants $d(i,j)$ which record the Euler characteristic of the rank $1$ $S_{2}$ sheaf on $E$ obtained by restricting $iK_{Y} + jE$.  \cite[Lemma 2.5]{Kawakita05} shows that when $i \frac{a}{n} + j \leq \frac{a}{n}$ then the higher cohomology of this sheaf vanishes and its sections are curves of anticanonical degree $-(i \frac{a}{n} + j)\frac{a}{n}E^{3}$.  We will freely use this and other notations from \cite{Kawakita05}.

\begin{lemm} \label{lemm:exceptionaltype}
Let $P$ be a germ obtained by a divisorial contraction of exceptional type from a terminal threefold.  Then either
\begin{itemize}
\item there is a family of (possibly reducible) curves $C$ in $E$ such that $-K_{Y} \cdot C \leq 1$ and $C$ deforms in dimension at least $1$ in $E$, or;
\item there is a family of (possibly reducible) curves $C$ in $E$ such that $-K_{Y} \cdot C \leq 2$ and $C$ deforms in dimension at least $2$ in $E$, or;
\item $P$ is an e3 singularity with $a=3$ and $n=1$, in which case $E$ carries a curve $C$ satisfying $-K_{Y} \cdot C \leq 1$.
\end{itemize}
\end{lemm}

\begin{proof}
We separate the proof into several cases.  
First suppose that $P$ is an exceptional singularity that occurs in a finite set of types (that is, any exceptional type except e1, e2, e7, e13).  Since there are only finitely many possible invariants for such singularities, one can just check the claim case-by-case.  It turns out that in almost all cases there is some $i$ with $1 \leq i \leq 6$ such that $d(-i,0) \geq 2$ and $i (\frac{a}{n})^{2}E^{3} \leq 1$.  The values of $i$ are recorded in Table \ref{table:finiteexceptional}.

The one remaining case is an e3 singularity with $a=3$ and $n=1$.  In this case $d(0,-1) = 1$ which defines a curve which has anticanonical degree $\leq 1$.

{\centering

\begin{table}[ht]
\caption{Finite families of exceptional singularities}
 \label{table:finiteexceptional}
\begin{tabular}{c|c|c}
singularity type & when $a=1$ & when $a>1$ \\ \hline
e3 & $n=1$ $\implies$ $i=1$, $n=3$ $\implies$ $i=2$ & ***  \\ \hline
e4 & not possible & not possible \\ \hline
e5 & $n=1$ $\implies$ $i=2$, $n=2$ $\implies$ $i=3$ & $a=2$, $n=1$ $\implies$ $i=1$  \\ \hline
e6 & $i = 3$ & not possible \\ \hline
e8 & $n=1$ $\implies$ $i=2$, $n=3$ $\implies$ $i=4$  & not possible \\ \hline
e9 & $n=1$ $\implies$ $i=3$, $n=2$ $\implies$ $i=5$  & $a=2$, $n=1$ $\implies$ $i=3$ \\ \hline
e10 & $i = 5$ & not possible \\ \hline
e11 & not possible & $a=2$, $n=2$ $\implies$ $i=4$ \\ \hline
e12 & $i = 4$ & not possible \\ \hline
e14 & $i = 3$ & not possible \\ \hline
e15 & $i = 4$ & not possible \\ \hline
e16 & $i = 6$ & not possible
\end{tabular}
\end{table}

}

We next suppose that $P$ is one of the exceptional singularity types that forms an infinite family (e1, e2, e7, e11).  In most cases (i.e.~when $r$ is sufficiently large) there is some $i$ with $1 \leq i \leq 4$ such that $d(-i,0) \geq 2$ and $i (\frac{a}{n})^{2}E^{3} \leq 1$.  Such situations are summarized in Table \ref{table:infiniteexceptional}.

{\centering

\begin{table}[ht] 
\caption{Infinite families of exceptional singularities}
\label{table:infiniteexceptional}
\begin{tabular}{c|c|c|c}
singularity type & $a$ value & $n$ value & $i$ value \\ \hline
\multirow{6}{*}{e1} & \multirow{2}{*}{1} & 1,2 & $i=1$ or $2$  ($r \geq 4$)  \\ \cline{3-4}
 & & 4 & $i=2$  \\ \cline{2-4}
 & \multirow{2}{*}{2}  & 1 & $i=1$ ($r \geq 9$) \\ \cline{3-4}
 &  & 2 & $i=1$ or $2$ ($r \geq 5$) \\ \cline{2-4}
 & \multirow{2}{*}{4} & 1 & $i=1$ ($r \geq 17$)  \\ \cline{3-4}
 & & 2 & $i =1$ or $2$ ($r \geq 9$)  \\ \hline
\multirow{3}{*}{e2} & \multirow{2}{*}{1}  & 1 & $i=1$ \\ \cline{3-4}
 &  & 2 & $i = 2$ \\ \cline{2-4}
 & 2 & 1,2 & $i=1$ or $2$ ($r \geq 4$) \\ \hline
e7 & 1 & 2 & $i=4$ \\ \hline
e13 & 1 & 1 & $i = 2$
\end{tabular}
\end{table}

}

Note that for an e2 singularity with $r < 4$, the possibility $(a,n) = (2,2)$ is ruled out by \cite[Lemma 3.3]{Kawakita05}.  Thus the only remaining case with an e2 singularity is when $r=3$, $a=2$, $n=1$.  In this situation  $d(0,-3) = 4$ and $3(\frac{a}{n})E^{3} = 2$.

There are also several remaining cases are of e1 type such that $r$ is small.  In the case when $r=3$, the possibilities $(a,n) = (2,2), (4,1), (4,2)$ are ruled out by \cite[Lemma 3.3]{Kawakita05}.  In the case when $r=5$, the possibility $(a,n) = (4,2)$ is ruled out by \cite[Lemma 3.3]{Kawakita05}.  When $(a,n) = (4,1)$ then $r$ must be congruent to $3$ or $5$ mod $8$ by \cite[Erratum]{Kaw03}.  \cite[Theorem 2.3]{Yamamoto18} rules out the case when $r=3$ and $(a,n) = (2,1)$.  \cite[Theorem 2.1 and Theorem 2.2]{Yamamoto18} rules out the case when $r=5$ and $(a,n) = (4,1)$.  Otherwise we use the construction summarized by Table \ref{table:e1smallr}.  This finishes the proof of the claim for exceptional type contractions.

{\centering

\begin{table}[ht] 
\caption{e1 type singularities with small $r$ value}
\label{table:e1smallr}
\begin{tabular}{c|c|c|c|c}
$r$ value & $a$ value & $n$ value & $d(i,j)$ & intersection number \\ \hline
\multirow{2}{*}{3} & \multirow{2}{*}{1} & 1  &  $d(-1,0) = 3$ & 4/3 \\ \cline{3-5}
& & 2  &  $d(-1,0) = 2$ & 2/3 \\ \cline{1-5}
 \multirow{1}{*}{5} & 2 & 1  &  $d(-1,0) = 3$ & 8/5 \\ \cline{1-5}
6 & 4 & 2  & $d(-1,0) = 3$ & 4/3 \\ \hline
 \multirow{2}{*}{7} & 2 & 1  & $d(0,-3) = 3$  & 12/7 \\ \cline{2-5}
  & 4 & 2 & $d(0,-3) = 3$ & 12/7 \\ \hline
11 & 4 & 1  & $d(0,-5) = 3$  & 20/11 \\ \hline
13 & 4 & 1  & $d(0,-6) = 3$  & 24/13
\end{tabular}
\end{table}

}
\end{proof}

\begin{lemm} \label{lemm:ordinarytype}
Let $P$ be a germ obtained by a divisorial contraction of ordinary type from a terminal threefold.  Then 
\begin{itemize}
\item if $P$ has o1 type then there is either a family of (possibly reducible) curves $C$ in $E$ such that $-K_{Y} \cdot C \leq 1$ and $C$ deforms in dimension at least $1$ in $E$ or there is a family of (possibly reducible) curves $C$ in $E$ such that $-K_{Y} \cdot C \leq 2$ and $C$ deforms in dimension at least $2$ in $E$;
\item if $P$ has o2 type then either $E$ carries a (possibly reducible) curve $C$ such that $-K_{Y} \cdot C \leq 1$ or there is a family of (possibly reducible) curves $C$ in $E$ such that $-K_{Y} \cdot C \leq 2$ and $C$ deforms in dimension at least $2$ in $E$;  
\item if $P$ has o3 type then $E$ carries a (possibly reducible) curve $C$ such that $-K_{Y} \cdot C \leq 1$.
\end{itemize}
\end{lemm}

\begin{proof}
First we consider the o1 case.  Suppose that $n \geq 2$. According to \cite[Theorem 3.7]{Kawakita05} $a$ and $n$ are coprime. Then \cite[Corollary 2.4]{Kawakita05} implies that $a = 1$. Thus we conclude that $(a/n)^2E^3 \leq 1$. Moreover we have $d(-1, 0) > 1$.  When $n =1$, $P$ is Gorenstein. According to \cite[Theorem 1.4]{Kaw03} we are in the situations of O or IV and this means that $b = 1$.  In case our singularity has type O, we have $\left(\frac{a}{n} \right)^2E^{3} =  2$ and $d(-1, 0) = 3$.  When our singularity has type IV, we have $\frac{a}{n}bE^3 = 2$ and $d(0, -1) = 3$. This proves our assertion for the o1 case.

Next we consider the o2 case.  First suppose that $\frac{a}{n} \leq 1$.  Then $d(0,-1) \geq 3$ so we obtain a covering family of curves of anticanonical degree $\leq 2$ that deforms in dimension $\geq 2$.  Next suppose that $\frac{a}{n} > 1$ and $n \geq 2$.  Then the proof of \cite[Theorem 4.5]{Kawakita05} shows that there are integers $s,t$ such that $as + nt = -1$.  The argument also shows that $d(s,t) \geq 1$ and the corresponding curve has anticanonical degree $\leq 1$.

Next suppose that we are in the o2 case with $n=1$ and $a>1$.  In this case \cite[Theorem 1.2]{Kawakita05} shows that $E$ can be realized as the hypersurface in the weighted projective space $\mathbb{P}(1,r,a,1)$ (where $1+r$ is divisible by $a$) defined by the weight $1+r$ equation
\begin{equation*}
x_{1}x_{2} + g(x_{3},x_{4}) = 0
\end{equation*}
such that the monomial $x_{3}^{(1+r)/a}$ occurs in $g$ with non-zero coefficient.  Consider the union of curves given by $V(g) \cap V(x_{1})$.  Each component $C$ will be defined by the equations $x_{1}=0$ and $x_{3} = \lambda x_{4}^{a}$ for some $\lambda$, and such a curve will have $H$-degree $\frac{1}{r}$.  Since the restriction of $H$ to $E$ coincides with $-E|_{E}$ when $E$ is embedded in $Y$, we see that $-K_{Y} \cdot C = \frac{a}{r} < 1$.

Finally we consider the o3 case.  Then we have $d(0,-1) \geq 1$, yielding a curve of anticanonical degree $\leq 1$.
\end{proof}

We are now prepared to prove Theorem \ref{theo:iitakadim0case}.  We separate into two cases based on whether or not $X$ admits an E5 divisor.

\begin{theo}
Let $X$ be a smooth Fano threefold which does not admit an E5 divisor.  There are no $a$-covers $f: Y \to X$ such that $Y$ is smooth and $\kappa(K_{Y}-f^{*}K_{X}) = 0$.
\end{theo}

\begin{proof}
We first run the relative $K_{Y}$-MMP for $f: Y \to X$.  The result is a birational contraction $\phi: Y \dashrightarrow Y'$ such that $Y'$ is a normal variety with terminal $\mathbb{Q}$-factorial singularities equipped with a morphism $f': Y' \to X$ such that $K_{Y'}$ is $f'$-relatively nef.

Note that every step of the relative $K_{Y}$-MMP is also a step of the $(K_{Y} - f^{*}K_{X})$-MMP.  Thus the ramification divisor $R' = K_{Y'} - f'^{*}K_{X}$ has Iitaka dimension $0$.  We will continue to run the $K_{Y'} - f'^{*}K_{X}$-MMP.  The end result of the MMP is a birational contraction $\psi: Y' \dashrightarrow \widetilde{Y}$ that contracts each component of $R'$.  

Consider the next step of the $K_{Y'} - f^{*}K_{X}$-MMP starting from $Y'$, denoted $\pi: Y' \to W$.  We will consider the possible types of this contraction and rule them out one by one:
\begin{enumerate}
\item Suppose $\pi: Y' \to W$ is a flipping contraction.  By \cite[Th\'eor\`eme 0]{Benveniste85} every curve $C$ contracted by the flipping contraction satisfies $-K_{Y'} \cdot C < 1$.  This implies that $-f'^{*}K_{X} \cdot C = 0$ so that $C$ is contracted by $f'$.  However, since $K_{Y'}$ is $f'$-nef this is an impossibility.

\item Suppose that $\pi: Y' \to W$ is a divisorial contraction taking the exceptional divisor $E$ to a curve $Z$.   Since terminal singularities have codimension $\geq 3$, we see that there is an open subset of $Z$ that is smooth and is contained in the smooth locus of $W$.  The fibers of the exceptional divisor over this locus are rational curves $C$ satisfying $-K_{Y'} \cdot C = 1$.  Thus they satisfy $-f'^{*}K_{X} \cdot C = 0$.  However, since $K_{Y'}$ is $f'$-nef it is not possible for such curves to have negative intersection against $K_{Y'} - f'^{*}K_{X}$.

\item Suppose that $\pi: Y' \to W$ is a divisorial contraction taking the exceptional divisor $E$ to a point $P$.  We first claim that $E$ cannot contain any curve $C$ satisfying $-K_{Y} \cdot C \leq 1$.  Indeed, in this case we must have $-f'^{*}K_{X} \cdot C = 0$ but as discussed above $K_{Y'} - f^{*}K_{X}$ does not have negative intersection with any curve contracted by $f'$.  We next claim that $E$ cannot contain a $\geq 2$-dimensional family of curves $C$ satisfying $-K_{Y'} \cdot C \leq 2$.  In this case these curves satisfy $-f'^{*}K_{X} \cdot C \leq 1$.  If the intersection number is $0$ we conclude as above.  If the intersection number is $1$ then Lemma \ref{lemm:linesarerational} shows that the images of the $C$ must sweep out an E5 divisor in $X$.  However, by assumption $X$ does not contain any such divisors.

According to Lemma \ref{lemm:exceptionaltype} and Lemma \ref{lemm:ordinarytype} every divisorial contraction will carry a curve $C$ satisfying one of the properties above, and thus no such contraction can occur.
\end{enumerate}
Altogether we see that there is no possible next step for the $K_{Y'} - f^{*}K_{X}$-MMP.  In other words, we see that the ramification divisor $R'$ must be equal to $0$.  Since $X$ is a smooth Fano variety this implies that $Y' = X$, showing that $X$ admits no $a$-covers.
\end{proof}

\begin{theo} \label{theo:iitakadim0casee5}
Let $X$ be a smooth Fano threefold which admits an E5 contraction.  There are no $a$-covers $f: Y \to X$ such that $Y$ is smooth and $\kappa(K_{Y}-f^{*}K_{X}) = 0$.
\end{theo}

\begin{proof}
By Theorem \ref{theo:e5classification} and Theorem \ref{theo:notbpf} if $X$ is a Fano threefold with an E5 contraction then $|-K_{X}|$ is basepoint free.  This implies that every curve on $X$ of anticanonical degree $1$ is a rational curve.  Thus, by the argument in the proof of \cite[Theorem 1.9]{LT16} (using \cite[Erratum]{LT16} in place of \cite[Proposition 7.2]{LT16}) we see that for a map $f: Y \to X$ as above every branch divisor for $f$ will be a rational surface swept out by $-K_{X}$-lines.

We next classify such divisors for the six Fano threefolds which admit an E5 contraction.  In cases (1-3) and (6), the only divisors which fit this description are the E5 divisors themselves.  In cases (4) and (5), in addition to the E5 divisor there are two additional divisors swept out by the reducible members of the unique conic fibration on $X$.  Let $U$ denote the complement of all such divisors, so that $f$ defines a torsion element in the fundamental group of the complement $U$.  Using the numbering system of Theorem \ref{theo:e5classification} we find the fundamental group of each of these complements and use it to verify that there are no such $a$-covers $f$.

\begin{itemize}
\item[(1)] $X$ is the blow up of $\mathbb{P}^{3}$ along a smooth cubic plane curve and the E5 divisor is the strict transform of the plane containing the cubic.  Let $U$ denote the complement of the E5 divisor.  Then $U$ contains an open subset $U'$ isomorphic to $\mathbb{A}^{3}$.  Since there is a surjection $\pi_{1}(U') \to \pi_{1}(U)$ we conclude that $U$ is simply connected.

\item[(2-3)] Each of these Fano threefolds admits a unique E5 divisor and is birational to a $\mathbb{P}^{1}$-bundle over $\mathbb{P}^{2}$.  Let $U \subset X$ denote the complement of the E5 divisor.  Using the birational map to the $\mathbb{P}^{1}$-bundle, in every case we see that $U$ contains an open subset $U'$ which is the complement of a section of a $\mathbb{P}^{1}$-bundle over $\mathbb{P}^{2}$.  Thus $U'$ and hence $U$ are simply connected.

\item[(4-5)] Each of these Fano threefolds admits a unique E5 divisor and is birational to a $\mathbb{P}^{1}$-bundle over $\mathbb{P}^{2}$.  In addition to the E5 divisor, there are two other rational surfaces swept out by lines, namely, the two divisors $D_{1},D_{2}$ which lie over the image of the blown-up curve in $\bP^{2}$.  Let $U \subset X$ denote the complement of the union of the E5 divisor, $D_{1}$, and $D_{2}$.  Using the birational map to the $\mathbb{P}^{1}$-bundle, we see that $U$ contains an open subset $U'$ which is the complement of a section in a $\mathbb{P}^{1}$-bundle over an open set $V \subset \bP^{2}$ which is the complement of a line or smooth conic.  In particular the fundamental group of $U'$ is isomorphic to the fundamental group of $V$.  This implies that any $a$-cover of $X$ is induced via base change by a cover $T \to \bP^{2}$.  While the corresponding maps $f: Y \to X$ will be $a$-covers they can never have Iitaka dimension $0$.

\item[(6)] The unique Fano threefold with more than one E5 structure is the blow-up of $\mathbb{P}_{\mathbb{P}^{2}}(\mathcal{O} \oplus \mathcal{O}(2))$ along a quartic curve contained in a minimal moving section of the projective bundle.  Let $U'$ denote the complement of the union of the rigid section with a minimal moving section in $\mathbb{P}_{\mathbb{P}^{2}}(\mathcal{O} \oplus \mathcal{O}(2))$.  Note that $U'$ is a toric variety; applying \cite[9.3 Proposition]{Danilov78} we see that $\pi_{1}(U') = \mathbb{Z}/2\mathbb{Z}$.    Let $U \subset X$ denote the complement of the E5 divisors.  Since $U'$ is an open subset of $U$, we see that $X$ admits at most one $a$-cover (up to birational equivalence).

We next verify that this ``potential'' $a$-cover is not actually an $a$-cover.  The non-trivial element of $\pi_{1}(U')$ defines a double cover $g: \mathbb{P}_{\mathbb{P}^{2}}(\mathcal{O} \oplus \mathcal{O}(1)) \to \mathbb{P}_{\mathbb{P}^{2}}(\mathcal{O} \oplus \mathcal{O}(2))$ whose ramification divisor is the union of the rigid section and a minimal moving section.  Let $Y$ be the blow-up of $\mathbb{P}_{\mathbb{P}^{2}}(\mathcal{O} \oplus \mathcal{O}(1))$ along a quartic curve in the non-rigid component of the ramification divisor.  Then $g$ lifts to a double cover $f: Y \to X$.  However, the ramification divisor of $f$ is big so that this map $f$ is not an $a$-cover.  We conclude that $X$ does not admit any $a$-cover as in the statement of the theorem.
\end{itemize}
\end{proof}

Together the previous two results prove Theorem \ref{theo:iitakadim0case}.  They also complete the last piece of Theorem \ref{theo:maintheorem2}.

\begin{proof}[Proof of Theorem \ref{theo:maintheorem2}:]
Consider a family of rational curves $M$ and let $\mathcal{C} \to M$ denote the universal family over $M$. Let $\beta: \widetilde{\mathcal C} \to \mathcal C$ be a resolution.  By assumption the evaluation map $ev: \widetilde{\mathcal{C}} \to X$ is dominant but does not have connected fibers.  \cite[Proposition 5.15]{LTCompos} shows that the evaluation map must factor rationally through a smooth $a$-cover $f: Y \to X$.  Furthermore, $M$ defines a dominant family of rational curves $C$ on $Y$ which satisfy $(K_{Y} - f^{*}K_{X}) \cdot C = 0$.

If $\kappa(K_{Y}-f^{*}K_{X}) = 2$, then Lemma \ref{lemm:iitakadim2case} shows that $Y$ is the base change of a family of conics on $X$.  The condition $(K_{Y} - f^{*}K_{X}) \cdot C = 0$ shows that $C$ is a finite cover of a general fiber of the Iitaka fibration for $K_{Y} - f^{*}K_{X}$, and thus its image in $X$ will be a conic.

If $\kappa(K_{Y}-f^{*}K_{X}) = 1$, then Theorem \ref{theo:iitakadim1case} shows that $Y$ is the base change of a del Pezzo fibration on $X$.  The condition $(K_{Y} - f^{*}K_{X}) \cdot C = 0$ shows that a general $C$ will be contracted by the Iitaka fibration for $K_{Y} - f^{*}K_{X}$.  Thus its image in $X$ will be contained in the fibers of the corresponding del Pezzo fibration.

Theorem \ref{theo:iitakadim0case} shows that $\kappa(K_{Y}-f^{*}K_{X}) = 0$ is impossible.
\end{proof}

\section{Movable Bend-and-Break Lemma}

In this section we prove the Movable Bend-and-Break Lemma for arbitrary smooth Fano threefolds.

\subsection{Multiple covers}

\cite[II.3.14 Theorem]{Kollar} shows that if a free curve $f: \mathbb{P}^{1} \to X$ has the property that $f^{*}T_{X}$ has at least two positive summands then a general deformation of $f$ will be an immersion.  In particular, such curves will have a locally free normal sheaf.  The following proposition identifies when this condition fails.

\begin{prop} \label{prop:onesummand}
Let $X$ be a smooth Fano threefold.  Let $M$ be a component of $\overline{\Rat}(X)$ generically parametrizing stable maps of anticanonical degree $\geq 3$ which have irreducible domain and map onto a free curve such that the restricted tangent bundle has exactly one positive summand.  Then the general map parametrized by $M$ is a multiple cover of a free curve of anticanonical degree $2$ in $X$.
\end{prop}

\begin{proof}
For a general map $f: \mathbb{P}^{1} \to X$ parametrized by $M$ write $f^{*}T_{X} = \mathcal{O}(a) \oplus \mathcal{O}^{2}$. Note that the tangent bundle of $\mathbb{P}^{1}$ must map to the $\mathcal{O}(a)$ factor.  Consider the sublocus of $M$ that sends a fixed point in $\mathbb{P}^{1}$ to a fixed general point of $X$.  Since $a>2$ this locus is positive dimensional, but the map can only deform in the tangent direction of the curve.  We conclude that $f$ must define a multiple cover of a curve $C$. Since such $C$ form a dominant family they must have anticanonical degree $\geq 2$.  If the anticanonical degree of $C$ is $d$ and the degree of the multiple cover is $r$, then the dimension of $M$ is $d + 2r - 2$.  Since $M$ is a component of $\overline{\Rat}(X)$ this must be at least the expected dimension $rd$, showing that $d = 2$.
\end{proof}

In particular, this implies that if $M \subset \overline{\Rat}(X)$ generically parametrizes free curves $f: \mathbb{P}^{1} \to X$ which are birational onto the image and have anticanonical degree $\geq 3$ then the general curve parametrized by $M$ will have locally free normal sheaf.

Proposition \ref{prop:onesummand} allows us to easily prove Movable Bend-and-Break for stable maps for which the restricted tangent bundle only has one positive summand.

\begin{prop} \label{prop:mbbmultiplecover}
Let $X$ be a smooth Fano threefold.  Let $M$ be a component of $\overline{\Rat}(X)$ generically parametrizing stable maps with irreducible domain that map onto a free curve such that the restricted tangent bundle has exactly one positive summand and the maps have anticanonical degree $\geq 3$.   Then $M$ contains a stable map $f : Z \rightarrow X$ such that $Z$ is the union of two $\mathbb P^1$'s and the restriction of $f$ to both components is free.
\end{prop}

\begin{proof}
By Proposition \ref{prop:onesummand} $M$ will parametrize multiple covers of free anticanonical conics.  Since the Kontsevich moduli space parametrizing degree $d$ maps $\mathbb{P}^{1} \to \mathbb{P}^{1}$ includes stable maps whose domains have two components, the statement follows by breaking the multiple covers of a fixed anticanonical conic.
\end{proof}

\subsection{Curves through general points}

We now focus on stable maps for which the restricted tangent bundle has at least two positive summands.  Starting from a component $M$ of $\overline{\Rat}(X)$, the idea is to consider the sublocus of $M$ parametrizing curves through many general points and meeting many general curves and to analyze how curves in this sublocus break.  The first step is to define precisely what we mean by a ``general curve'' in this context.

\begin{defi} \label{defi:basepointfree}
Let $X$ be a Fano threefold and let $p: U \to B$ be a family of irreducible reduced curves on $X$ with evaluation map $ev: U \to X$.  We say that $p$ is a basepoint free family if the evaluation map is flat.
\end{defi}

The key property of a basepoint free family of curves is that a general member will avoid any fixed codimension $2$ subset.  One way of constructing a basepoint free family of curves is to take intersections of general members of a very ample linear system $|H|$ on $X$.  In fact, this construction also yields a basepoint free family of curves when $H$ is only big and basepoint free.

\begin{lemm} \label{lemm:deformationcount}
Let $X$ be a smooth projective threefold.  Fix a basepoint free family of curves $p: U \to B$ on $X$.  Let $M$ be a component of $\overline{M}_{0,0}(X)$ generically parametrizing maps onto free curves such that the restricted tangent bundle has at least two positive summands. Suppose that the normal sheaf to the general map parametrized by $M$ has the form $\mathcal{O}(a) \oplus \mathcal{O}(b)$ with $a \leq b$.  Let $M^{\circ}$ denote the open subset parametrizing such maps.  Then for any positive integer $r \leq a+1$ and any non-negative integer $s \leq a+b+2-2r$ the sublocus $T_{r,s}$ of $M^{\circ}$ parametrizing curves which intersect $r$ general points and $s$ general members of $p$ satisfies
\begin{equation*}
\codim(T_{r,s}) \geq 2r + s.
\end{equation*}
If the basepoint free family of curves is constructed by taking general complete intersections of a big and basepoint free linear series on $X$, then this inequality is actually an equality.

If $r > a+1$ or $s > a+b+2-2r$ then $T_{r,s}$ is empty.
\end{lemm}

\begin{proof}
Suppose we fix $r$ points on a curve $C$ parametrized by our family.  Then the deformations of $C$ through the $r$ points will form a dominant family of curves on $X$ if and only if $N_{f/X}(-r)$ is globally generated.  This shows that deformations of $C$ can go through $a+1$ general points of $X$ but no more.  Since $H^1(N_{f/X}(-r)) = 0$ for $r \leq a+1$, we see that that the subvariety of $M^\circ$ parametrizing curves through $r$ general points is of codimension $2r$ for any $r \leq a+1$.

Suppose we fix any subvariety $W$ of $M^{\circ}$.  For any fixed curve parametrized by $M^{\circ}$ a general member $Q$ of the family $p$ will not intersect this curve.  In particular, this implies that the locus of curves in $W$ which intersect $Q$ has codimension $\geq 1$.  This verifies the first claim and the last claim.

Finally, suppose that $p$ parametrizes general complete intersections of a big and basepoint free linear series.  We must show that for any sublocus $W$ of $M$ parametrizing curves through $r$ general points and $s$ general members of $p$, the sublocus of $W$ which intersects an additional general member $Q$ of $p$ will have codimension exactly $1$.  By induction on $s$, we may assume that each component of $W$ has the expected dimension.  Thus, the curves parametrized by $W$ will sweep out either a surface $Y$ or all of $X$.  Furthermore, if $W$ sweeps out a surface $Y$ then $Y$ must contain a general point of $X$ and thus cannot be contracted by any generically finite morphism from $X$.  We deduce that $Q$ will intersect this surface in a finite number of (general) points.  This proves the statement.
\end{proof}

We now classify certain stable maps whose images contain many general points and curves on $X$.

\begin{prop} \label{prop:generalcurveclassification}
Let $X$ be a smooth Fano threefold.  Construct a basepoint free family of curves $p: U \to B$ on $X$ by taking general complete intersections of elements of a big and basepoint free linear series.  Let $M$ be a component of $\overline{\Rat}(X)$ generically parametrizing stable maps with irreducible domain that map birationally onto a free curve of anticanonical degree $\geq 5$.  Suppose that the normal sheaf of the general curve parametrized by $M$ has the form $\mathcal{O}(a) \oplus \mathcal{O}(b)$ with $a \leq b$.   Let $M^{\circ}$ denote the dense open subset parametrizing such maps.  Consider the following families in $M$:
\begin{itemize}
\item If $a<b$, consider the closure $T$ in $M$ of all one-dimensional components of the sublocus of $M^{\circ}$ parametrizing curves which contain $a+1$ general points and meet $b-a-1$ general members of $p$.
\item If $a=b$, consider the closure $T$ of all the one-dimensional components of the sublocus of $M^{\circ}$ parametrizing curves which contain $a$ general points and meet $1$ general member of $p$.
\end{itemize}
Suppose that $f: Z \to X$ is a stable map parametrized by $T$ such that the domain $Z$ is reducible and there are at least two free components of $f$.  Then $f$ must fall into one of the following categories.
\begin{enumerate}
\item $Z$ is the union of two $\mathbb P^1$'s and both curves map birationally to a free rational curve.
\item The image of $Z$ consists of a union of two disjoint free curves and one $-K_{X}$-line of E5 type which connects the two free curves.
\item The image of $Z$ consists of a union of free curves with two $-K_{X}$-lines of E5 type.
\item The image of $Z$ consists of a union of free curves with a $-K_{X}$-conic of E1, E3, E4, or E5 type.
\end{enumerate}
Furthermore, cases (3) and (4) can only occur if $a=b$ and in these cases one of the non-free curves must meet a general member of $p$.
\end{prop}

\begin{rema}
We briefly explain what ``general'' means in the context of Proposition \ref{prop:generalcurveclassification}.  First, we choose our points general so that they impose independent conditions on every family of curves of degree at most $a+b+2$.  We also ensure that they do not lie on any line or any contractible divisor on $X$.  We can also guarantee that for any family of free curves of degree $\leq a+b+2$ the points do not lie in the proper subset of $X$ swept out by non-free deformations of these curves.

Second, we choose our curves general so that they impose independent conditions on any family of curves of degree $\leq a+b+2$ and also on any sublocus of curves of degree $\leq a+b+2$ through some of our general points.  Consider the proper closed subset $W$ of $X$ formed by all non-free curves of degree $\leq a+b+2$; we also ensure that if the family of free curves of a given degree through some of our general points sweeps out a surface $Y$ then our curves avoid the $1$-dimensional subset $Y \cap W$.
\end{rema}

\begin{proof}[Proof of Proposition \ref{prop:generalcurveclassification}:]
By assumption $Z$ is reducible, so we can write $$Z =\sum_{i = 1}^r Z_i + (\text{contracted components})$$ with $r \geq 2$.
By assumption there are at least two components of $Z$ which map onto free curves.  After relabeling indices we may assume that $Z_{i}$ is free if and only if $1 \leq i \leq k$ where $k \geq 2$. For $1 \leq i \leq k$, we write $f_i' : Z_i'\to X$ for a general deformation of $f|_{Z_i} : Z_i \to X$ in its family.
Suppose that for $1 \leq i \leq k$ we have the normal sheaves 
\[
N_{f_i'/X} = \mathcal O(a_i) \oplus \mathcal O(b_i)\oplus(\mathrm{tor})
\]
with $a_{i} \leq b_{i}$. As mentioned before the torsion part vanishes unless $a_i = b_i = 0$. In particular $Z_i$ can pass through at most $a_i+1$ general points due to Lemma \ref{lemm:deformationcount}.

We let $d_{i}$ denote the anticanonical degree of $Z_{i}$.  By comparing anticanonical degrees we see that 
\begin{equation} \label{eq:comparingdegrees}
a+b+2 = \sum_{i=1}^{r} d_{i} \geq \sum_{i = 1}^k (a_i + b_i + 2).
\end{equation}

First suppose that $a=b$.  Thus we are considering a family of curves through $a$ general points and intersecting a general member $Q$ of our basepoint family of curves.  The image of each $Z_{i}$ for $1 \leq i \leq k$ can contain at most $a_{i} + 1$ general points of $X$, and thus
\begin{align} \label{eq:calculation1}
a & \leq \sum_{i=1}^{k} (a_{i} + 1) \\
\nonumber & \leq \sum_{i=1}^{k} \frac{d_{i}}{2} \\
\nonumber & = a + 1 - \frac{\sum_{i = k + 1}^rd_i}{2}
\end{align}
where the last equality follows from Equation \eqref{eq:comparingdegrees} and the fact that $a=b$.  This means that there are four possibilities for the non-free components of the image of $Z$:
\begin{enumerate}[label=(\roman*)]
\item There are two $-K_X$-lines.
\item There is one non-free $-K_X$-conic.
\item There is one $-K_X$-line.
\item There are no non-free components.
\end{enumerate}
We analyze each case separately.

In case (i), all the inequalities of Equation \eqref{eq:calculation1} must achieve the equality.  In particular, each free component $Z_{i}$ must go through exactly $a_{i}+1$ general points and have degree $2a_{i}+2$.  Thus there are only finitely many possibilities for each $Z_{i}$ and by generality we know that $Z_{i} \cap Z_{j} = \varnothing$ for $i \neq j$ (no matter which of the finite set of choices we pick).  To ensure that the total curve is connected, each $Z_{i}$ must meet one of the lines.  Also, one of the lines must meet $Q$.  Suppose that $L$ is a component of the moduli space of lines that is $1$-dimensional.  Then only finitely many lines parametrized by $M$ intersect each $Z_{i}$ and $Q$, and in particular, any line parametrized by $L$ can only intersect at most one of these curves.  Thus, we see that both lines must be of E5 type in order to connect the $\geq 2$ free components of the curve and to meet with $Q$.

In case (ii), the same argument shows that every free component must contain the maximum number of general points and the points determine a finite set of possibilities for each $Z_{i}$. Thus the $Z_{i}$ must be disjoint.  We claim that the conic must be of E1, E3, E4, or E5 type.  Otherwise the conic can move in at most dimension $2$, but in this situation it is impossible for the conic to connect the two free curves and to meet $Q$.

In case (iii), all but one of the free components $Z_{i}$ must contain $a_{i}+1$ general points and have degree $2a_{i}+2$.  The last component $Z_{k}$ must contain $a_{k}+1$ general points and have degree $2a_{k}+3$.  The general points determine a $1$-dimensional set of possible choices of $Z_{k}$ and a finite set of possibilities for each other $Z_{i}$.   Thus $Q$ must either meet with the $-K_X$-line $\ell$ or with $Z_{k}$.  First suppose that $Q$ meets with $Z_{k}$.  Then the choice of $a_{k}+1$ general points and $Q$ determine a finite set of possible choices for $Z_{k}$.  Arguing as above, the only way $Z$ can be connected is if there are exactly two free components which are connected by a line of E5 type.  Next suppose that $Q$ meets with $\ell$ and that $\ell$ does not have E5 type.  Then $Q$ determines a finite set of possible lines $\ell$.  But then it is impossible for the curve $Z_{k}$ to meet $\ell$ and to meet one of the other free curves $Z_{i}$.  We conclude that in this case as well the line must have E5 type.  Since the line meets $Q$, it can only intersect one other component of type $Z_{i}$, and in this case $Z$ can only have two free components.

In case (iv), all but two of the free components $Z_{i}$ must contain $a_{i}+1$ general points and have degree $2a_{i}+2$.  For the last two components $Z_{k-1},Z_{k}$ either
\begin{itemize}
\item we have $a = \sum (a_{i}+1)$, $Z_{k-1}$ contains $a_{k-1}+1$ general points and has degree $2a_{k-1}+3$ and $Z_{k}$ contains $a_{k}+1$ general points and has degree $2a_{k}+3$, or
\item we have $a = \sum (a_{i}+1)$, $Z_{k-1}$ contains $a_{k-1}+1$ general points and has degree $2a_{k-1}+2$ and $Z_{k}$ contains $a_{k}+1$ general points and has degree $2a_{k}+4$, or
\item we have $a < \sum (a_{i}+1)$, $Z_{k-1}$ contains $a_{k-1}+1$ general points and has degree $2a_{k-1}+2$ and $Z_{k}$ contains $a_{k}$ general points and has degree $2a_{k}+2$.
\end{itemize}
In the first case, one of $Z_{k-1},Z_{k}$ must meet $Q$, and this will restrict the number of options for this curve to a finite set.  In order to obtain a connected curve, the only possibility is that our curve has exactly two components. In the second case, $Z_{k}$ must meet $Q$ and this will determine a $1$-dimensional family of possible curves.  Thus again in order to be connected we must have only two components.  In the third case, the deformations of $Z_{k}$ through the $a_{k}$ general points will form a dominant family of curves.  In order to meet $Q$ and have a connected image we must have only two components.  This finishes the argument in the case when $a=b$.

Assume now that $b \geq a+1$.  In this case we are considering a family through $a+1$ general points $p_1, \dots, p_{a+1}$ and $b-a-1$ general curves $Q_1, \dots, Q_{b-a-1}$ in the basepoint free family $p$.  Let $Y$ be the closed set of $X$ swept out by all the non-free rational curves of degree at most $a+b+2$ and let $\Sigma$ be the surface swept out by the
rational curves parametrized by $M^{\circ}$ which contain $p_1, \dots, p_{a+1}$. Note that we have $\dim Y \leq 2$.  

Note that the non-free components of $Z$ lie on the intersection of $Y$ and $\Sigma$ and that this intersection is a proper subset of $\Sigma$.  Since we choose the $p_{i}$ general none of them lie on $Y$, and in particular a non-free component of $Z$ cannot meet any of the $p_{i}$.  Since we choose the $Q_{i}$ general with respect to the $p_{i}$ (and hence with respect to $Y$), we can ensure that the intersection points with $\Sigma$ do not lie in $Y \cap \Sigma$, guaranteeing that a non-free component of $Z$ cannot intersect any of the $Q_i$.

Assume that for $1\leq i \leq k$ the curve $Z_i$ passes through $\alpha_i$ of the points and intersects $\beta_i$ of the curves. By Lemma \ref{lemm:deformationcount} we have
\begin{equation*}
\sum_{i=1}^k d_i \geq \sum_{i=1}^k 2\alpha_i+ \beta_i \geq 2(a+1)+(b-a-1)=a+b+1.
\end{equation*}
Since the total degree of $Z$ is $a+b+2$, there can be at most one non-free component in $Z$ and if there is such a component it must have degree $1$. In this case the above inequality has to be an equality and therefore $d_i = 2\alpha_i+\beta_i$ for each $1\leq i \leq k=r-1$.  This means that the general points and curves determine a finite number of possibilities for each of the free components $Z_{i}$. In particular the distinct free components of $Z$ do not intersect.  By assumption there are at least 2 such components, so the line must intersect both of these curves. This implies that the line must be of E5 type and there can only be two free components in $Z$.

Finally, if there is no non-free component of $Z$, then all free components of $Z$ but one are determined up to a finite set by the general points and curves that they meet.  There is one component of $Z$ which can deform in dimension $\leq 1$.  Thus by arguing as before we conclude that $Z$ has exactly two free components.
\end{proof}

\subsection{No E5 contraction}

Movable Bend-and-Break has a particularly clean proof for Fano threefolds which do not admit an E5 contraction.

\begin{theo} \label{theo:mbbfanothreefolds}
Let $X$ be a smooth Fano threefold that does not admit an E5 contraction.  Let $M$ be a component of $\overline{\Rat}(X)$ that generically parametrizes free curves.  Suppose that the curves parametrized by $M$ have anticanonical degree $\geq 5$.
Then $M$ contains a stable map $f : Z \rightarrow X$ such that $Z$ is the union of two $\mathbb P^1$'s and the restriction of $f$ to both components of $Z$ is free.
\end{theo}

\begin{proof}
By Proposition \ref{prop:onesummand} and Proposition \ref{prop:mbbmultiplecover} it suffices to consider the case when the general map parametrized by $M$ has a restricted tangent bundle with at least two positive summands.  In particular we may suppose that the general map is an immersion and that the normal sheaf is locally free.

Consider the set of contractible divisors on $X$ which have E3, E4 type or which have E1 type and are isomorphic to $\mathbb{P}^{1} \times \mathbb{P}^{1}$ with normal bundle $\mathcal{O}(-1,-1)$.  Lemma \ref{lemm:nooverlap} shows that all such divisors are disjoint.  Thus there is a birational map $\phi: X \to W$ contracting all such divisors.

Let $p: U \to B$ be a basepoint free family of curves constructed by taking complete intersections of the pullback of a very ample divisor on $W$.  Then a general member of $p$ does not intersect any $\phi$-exceptional divisor.  Suppose that $T$ is the closure of a one-parameter subfamily of curves $C$ as in the statement of Proposition \ref{prop:generalcurveclassification}.  (Note that the existence of such a locus $T$ is guaranteed by Lemma \ref{lemm:deformationcount}.)  As in Proposition \ref{prop:generalcurveclassification} we will write $\mathcal{O}(a) \oplus \mathcal{O}(b)$ for the normal sheaf of $C$.

We claim that $T$ parametrizes a stable map $f: Z \to X$ with reducible domain and such that at least two components of $Z$ are free.  First assume that $a>0$.  Then the curves parametrized by $T$ will go through two general points $x_{1},x_{2}$ of $X$.  By applying Lemma~\ref{lemm:BandB} with the chosen points $x_{1},x_{2}$, we see that $T$ contains a stable map $f: Z \to X$ with reducible domain and such that at least two components of $Z$ are free.  Next suppose that $a=0$.  Then we must have $b \geq 3$.  Let $\Sigma$ be the surface swept out by deformations of $C$ through a general point $x_{1}$.  Since $x_{1}$ is general, a general point $x_{2}$ in $\Sigma$ will not lie on any non-free curve of anticanonical degree less than the anticanonical degree of $C$.  Since the curves parametrized by $T$ must meet one general point and a general member of $p$, they will need to go through a general point on $\Sigma$.  Thus when we apply Lemma~\ref{lemm:BandB} with the chosen points $x_{1},x_{2}$ we again obtain a stable map with reducible domain and at least two free components. 

We may now apply Proposition \ref{prop:generalcurveclassification} to the stable map $f: Z \to X$.  Since $X$ does not contain any E5 divisors by assumption, the proposition guarantees that any non-free components of $Z$ must lie in an exceptional divisor contracted by $\phi$.  In particular, such components cannot meet a general element of the family $p$.  Thus cases (2), (3), (4) of Proposition \ref{prop:generalcurveclassification} cannot occur and we deduce the desired statement.
\end{proof}

\subsection{E5 contractions} \label{sect:e5contractions}
We next turn to Fano threefolds which admit an E5 contraction.  By the classification \cite{MM81} and \cite{MM03} and the computation of extremal rays in \cite{Mat95}, \cite{MM04}, and \cite[Section 10.4]{Fuj16} there are six such Fano threefolds.  We note that for these threefolds Movable Bend-and-Break may no longer hold for families of curves with anticanonical degree $5$.

\begin{exam}
Consider the Fano threefold $\mathbb{P}_{\mathbb{P}^{2}}(\mathcal{O} \oplus \mathcal{O}(2))$.  The minimal family of moving sections gives a dominant family of $\mathbb{P}^{2}$'s on $X$.  The lines in these $\mathbb{P}^{2}$'s give a dominant family of movable curves of anticanonical degree $5$ which do not break into a union of two free curves.  Rather, as in Proposition \ref{prop:generalcurveclassification}.(2) these curves will break into the union of two fibers of the $\mathbb{P}^{1}$-bundle structure connected by a line in the rigid divisor of E5 type.
\end{exam}

Instead, the correct bound is given by the following theorem.

\begin{theo} \label{theo:mbbe5case}
Let $X$ be a smooth Fano threefold which admits an E5 contraction.  Let $M$ be a component of $\overline{\Rat}(X)$ generically parametrizing stable maps with irreducible domain that map birationally onto a free curve in $X$.  Suppose that the curves parametrized by $M$ have anticanonical degree $\geq 6$.
Then $M$ contains a stable map $f : Z \rightarrow X$ such that $Z$ is the union of two $\mathbb P^1$'s and the restriction of $f$ to both components of $Z$ is free.
\end{theo}

We prove this statement for the hardest case -- the unique Fano threefold admitting two E5 divisors -- in Claim \ref{clai:mbbinhighdegree}.  For the remaining five cases, we will use the following argument.

\begin{prop} \label{prop:e5contractioninduction}
Let $X$ be a smooth Fano threefold which admits a unique divisor $E$ of E5 type.  Assume that every free rational curve $T$ of anticanonical degree at most $4$ satisfies $E \cdot T \leq 1$.  Then any family of free rational curves $C$ of anticanonical degree $\geq 5$ will satisfy Movable Bend-and-Break except possibly families which satisfy $E \cdot C = 0$, $-K_{X} \cdot C \leq 9$, and contain a stable map whose image consists of two disjoint free curves of anticanonical degree $\leq 4$ connected via a line in the E5 divisor. 

In particular, Movable Bend-and-Break holds for all families of anticanonical degree $\geq 10$.
\end{prop}

\begin{proof}
By Proposition \ref{prop:onesummand} and Proposition \ref{prop:mbbmultiplecover} it suffices to consider the case when the general map parametrized by $M$ has a restricted tangent bundle with at least two positive summands.  In particular we may suppose that the general map is an immersion and that the normal sheaf is locally free.

By Lemma \ref{lemm:nooverlap} there is a birational map $\phi: X \to W$ contracting all E3, E4, E5 divisors and all E1 divisors isomorphic to $\mathbb{P}^{1} \times \mathbb{P}^{1}$ with normal bundle $\mathcal{O}(-1,-1)$.  Let $p: U \to B$ be the basepoint free family of curves constructed by taking complete intersections of the pullback of a very ample divisor on $W$.  Then a general member of $p$ does not intersect any $\phi$-exceptional divisor.

Consider a family of free rational curves $C$ of anticanonical degree $d \geq 5$.  The proof is by induction on $d$.   Assume that this family does not satisfy the criteria of Movable Bend-and-Break.  Suppose that $T$ is a one-parameter subfamily of curves $C$ as in the statement of Proposition \ref{prop:generalcurveclassification}.  Note that our choice of basepoint free curve rules out cases (3) and (4) of Proposition \ref{prop:generalcurveclassification}.  Since we are assuming (1) does not hold, the only option is that $C$ breaks into a union $C_{1} \cup \ell \cup C_{2}$ where $C_{1},C_{2}$ are disjoint free curves and $\ell$ is a line in $E$ that connects $C_{1}$ and $C_{2}$.

For $i=1,2$ the curve $C_{i}$ has non-vanishing intersection with $E$.  By induction, it either satisfies $-K_{X} \cdot C_{i} \leq 4$ and $E \cdot C_{i} = 1$ or its family satisfies Movable Bend-and-Break.  First assume that both $C_{1}$ and $C_{2}$ have anticanonical degree at most $4$.  Then $-K_{X} \cdot C \leq 9$ and
\begin{equation*}
E \cdot C = E \cdot (C_{1} + \ell + C_{2}) = 0.
\end{equation*}
Otherwise, without loss of generality we may assume that $C_{1}$ satisfies Movable Bend-and-Break.  Note that the two components $C_{1}$ and $C_{2}$ must be general in their deformation classes by construction.  Thus, Proposition \ref{prop:transverseintersection} guarantees that they intersect $E$ transversally.  We deduce that the stable map corresponding to our broken curve is a smooth point of $\overline{M}_{0,0}(X)$ by Proposition \ref{prop:vanishingofH1} and is thus contained in a unique component.  This component also contains a stable map whose image has the form $(C_{1}' \cup C_{1}'') \cup \ell \cup C_{2}$ where $C_{1}' \cup C_{1}''$ is a general union of free curves obtained by applying Movable Bend-and-Break to $C_{1}$.  Note that this new stable map defines a smooth point of $\overline{M}_{0,0}(X)$.  By Proposition \ref{prop:vanishingofH1} we can smooth the subcurve $C_{1}'' \cup \ell \cup C_{2}$, and the resulting stable map will lie in our original component of $\overline{M}_{0,0}(X)$, verifying Movable Bend-and-Break for this component.
\end{proof}

\begin{prop}
Let $X$ be a smooth Fano threefold which admits a unique E5 divisor.  Then Movable Bend-and-Break holds for free rational curves of anticanonical degree $\geq 6$.
\end{prop}

Together with Claim \ref{clai:mbbinhighdegree} this proves Theorem \ref{theo:mbbe5case}.  Theorem \ref{theo:maintheorem3} follows immediately from Theorem \ref{theo:mbbe5case}, Theorem \ref{theo:mbbfanothreefolds}, and Proposition \ref{prop:mbbmultiplecover}.

\begin{proof}
For each of the five types of Fano threefold as in Theorem \ref{theo:e5classification} we prove two things.  First, we check that every free rational curve of anticanonical degree $\leq 4$ has intersection at most $1$ against the E5 divisor.  This verifies the hypotheses for Proposition \ref{prop:e5contractioninduction}.  Second, we verify  Movable Bend-and-Break by hand for families of rational curves with anticanonical degree between $6$ and $9$ which have vanishing intersection against the E5 divisor.  Together these verify the desired statement.

\begin{enumerate}
\item Suppose $X$ is the blow up of $\mathbb{P}^{3}$ along a smooth cubic plane curve $Z$.  The E5 divisor $E$ is the strict transform of the plane $P$ containing $Z$.  We let $H$ denote the pullback of the hyperplane class on $\mathbb{P}^{3}$.  Then $-K_{X} = 3H + E$.  Thus any free rational curve of anticanonical degree $\leq 4$ must map to a line in $\mathbb{P}^{3}$ and satisfy $E \cdot C \leq 1$.

We must verify Movable Bend-and-Break for the families of rational curves $C$ which satisfy $E \cdot C = 0$ and $H \cdot C = 2$ or $3$.  These are families of conics or cubics in $\mathbb{P}^{3}$ which meet $P$ only along $Z$.  It is easy to see that in either case we can break off a line in $\mathbb{P}^{3}$ meeting $P$ only along $Z$.

\item Suppose $X$ is $\mathbb{P}_{\mathbb{P}^{2}}(\mathcal{O} \oplus \mathcal{O}(2))$.  The E5 divisor $E$ is the rigid section of the projective bundle.  We let $H$ denote the pullback of the hyperplane class on $\mathbb{P}^{2}$.  Then $-K_{X} = 2E + 5H$.  Any free rational curve of anticanonical degree $\leq 4$ must be contracted by the $\mathbb{P}^{1}$-bundle map.  Thus the only such curves are the fibers of the projective bundle.

There are no families of curves that satisfy $E \cdot C = 0$ and $6 \leq -K_{X} \cdot C \leq 9$.

\item Suppose $X$ is the blow up of $\mathbb P^3$ along the disjoint union of a plane cubic $Z$ in a plane $P$ and a point $Q$ not on $P$. The E5 divisor $E$ is the strict transform of $P$.  Let $E_{0}$ denote the exceptional divisor lying above the point.  Let $p: X \to \mathbb{P}^{2}$ be the composition of the birational map to $\mathrm{Bl}_{pt}\mathbb{P}^{3}$ with the projective bundle map and let $H$ denote the pullback of the hyperplane class on $\mathbb{P}^{2}$.  We have $-K_{X} = E_{0} + E + 3H$.  Thus the free rational curves of anticanonical degree $\leq 4$ will either be fibers of $p$ or will satisfy $E \cdot C \leq 1$.  

We next discuss the remaining cases with anticanonical degree between $6$ and $9$.  If $E_{0} \cdot C = 0$ then the corresponding family of curves is the strict transform of a family of rational curves from (1) and we have already verified Movable Bend-and-Break.  Any other irreducible moving curve which satisfies $E \cdot C = 0$ and has anticanonical degree between $6$ and $9$ will satisfy either:
\begin{enumerate}
\item $E_{0} \cdot C = 1$, $H \cdot C = 2$, $E \cdot C = 0$.
\item $E_{0} \cdot C \geq 2$, $H \cdot C = 2$, $E \cdot C = 0$.
\end{enumerate}
However (a) is the only one which is also numerically equivalent to the sum of an E5 line with two free curves of anticanonical degree $\leq 4$.  Thus it suffices to consider this case.

Note that the blow-up of $\mathbb{P}^{3}$ at a point has a $\mathbb{C}^{*}$-action fixing the exceptional divisor and the plane $P$.  This induces a $\mathbb{C}^{*}$-action on $X$.  Let $C$ be a general curve in one of the families above and use the $\mathbb{C}^{*}$-action to take a limit in the direction of the exceptional divisor.  Then the limit curve will be the union of a conic in $E_{0}$ with three $p$-vertical curves which meet the exceptional divisor over $Z$ but are not contained in it.

Suppose we deform the conic in $E_{0}$ while fixing these three points plus one more intersection point with the curve in $E_{0}$ beneath $Z$.  This one-parameter family of curves will break into a curve which is the union of two lines in $E_{0}$ with the same three or four vertical curves.  Furthermore, since at most two distinct points on a conic can be collinear, each line will meet at most two of these vertical fibers.  We can then smooth the union of a line with the vertical curves it meets to obtain two free curves.  All the stable maps obtained in this construction are smooth, showing that the resulting union of two free curves is in the same component of $\overline{M}_{0,0}(X)$ we started with.

\item Suppose $X$ is the blow up of $\mathbb{P}^{1} \times \mathbb{P}^{2}$ along a conic $Z$ in a fiber $F_{0}$ of the first projection map.  The E5 divisor $E$ is the strict transform of the plane containing the blown-up conic.  Let $H_1$ be the pullback of the hyperplane class from $\bP^1$ and let $H_2$ be the pullback of the hyperplane class from $\bP^2$.  We have $-K_{X} = H_1 + 3H_2 + E$. Thus the free rational curves of anticanonical degree $\leq 4$ will either be contained in a fiber of $X \to \mathbb{P}^{2}$ or will satisfy $E \cdot C = 0$.

By Proposition \ref{prop:e5contractioninduction} Movable Bend-and-Break holds for all families of degree $\geq 5$ which do not contain a reducible curve of the form $C_{1} \cup \ell \cup C_{2}$ where $C_{1},C_{2}$ are fibers of the map to $\bP^{2}$ and $\ell$ is a line in the E5 divisor.  Since this curve has anticanonical degree $5$, we conclude that Movable Bend-and-Break holds in all degrees $\geq 6$.

\item Fix a point $p \in \mathbb{P}^{3}$ and suppose $X$ is the blow up of $\mathrm{Bl}_{p}\mathbb{P}^{3}$ along a line $\ell$ contained in the exceptional divisor.  The E5 divisor $E$ is the strict transform of the exceptional divisor over $p$.    Let $E_{*}$ denote the exceptional divisor for the map to $\mathrm{Bl}_{p}\mathbb{P}^{3}$.  Let $p: X \to \mathbb{P}^{2}$ be the composition of birational map to $\mathrm{Bl}_{p}\mathbb{P}^{3}$ with the projective bundle map and let $H$ denote the pullback of the hyperplane class on $\mathbb{P}^{2}$.  We have $-K_{X} = 4H + 2E + E_{*}$.  Thus the free rational curves of anticanonical degree $\leq 4$ will either be a fiber of $X \to \mathbb{P}^{2}$ or will satisfy $E \cdot C = 0$. 

By Proposition \ref{prop:e5contractioninduction} Movable Bend-and-Break holds for all families of degree $\geq 5$ which do not contain a reducible curve of the form $C_{1} \cup \ell \cup C_{2}$ where $C_{1},C_{2}$ are fibers of the map to $\bP^{2}$ and $\ell$ is a line in the E5 divisor.  Since this curve has anticanonical degree $5$, we conclude that Movable Bend-and-Break holds in all degrees $\geq 6$.
\end{enumerate}
\end{proof}

\section{Geometric Manin's Conjecture} \label{sec:maninsconj}

Let $X$ be a smooth Fano threefold.  In this section we explain how to interpret Geometric Manin's Conjecture for Fano threefolds.  We then use Movable Bend-and-Break to obtain a polynomial bound on the number of components of $\overline{\Rat}(X)$ as predicted by Batyrev.  This yields an asymptotic upper bound on the counting function in Geometric Manin's Conjecture.

\subsection{The exceptional set}

Let $X$ be a smooth Fano variety equipped with the anticanonical divisor $-K_{X}$.  When $X$ is defined over a number field, \cite[Section 5]{LST18} gives a conjectural description of the exceptional set for Manin's Conjecture using the geometry of the $a$ and $b$ invariants.  Loosely speaking, one should remove the contributions of any generically finite morphism $f: Y \to X$ such that the $a,b$ invariants for $Y$ are larger than for $X$.  In our setting, this exceptional set should be interpreted as follows:

\begin{defi} \label{defi:manincomponent}
Let $M \subset \overline{\Rat}(X)$ be a component.  Suppose that $f: Y \to X$ is a generically finite morphism from a smooth projective variety $Y$ with one of the following types:
\begin{enumerate}
\item $f: Y \to X$ satisfies $a(Y,-f^{*}K_{X}) > a(X,-K_{X})$.
\item $f: Y \to X$ is dominant, $\kappa(K_{Y} - f^{*}K_{X}) > 0$, and
\begin{equation*}
(a(Y,-f^{*}K_{X}),b(Y,-f^{*}K_{X})) \geq  (a(X,-K_{X}), b(X,-K_{X}))
\end{equation*}
in the lexicographic order. 
\item $f: Y \to X$ is dominant, face contracting, and satisfies $\kappa(K_{Y} - f^{*}K_{X}) = 0$. 
\end{enumerate}
Suppose that for such a map $f: Y \to X$ there is a component $N$ of $\overline{\Rat}(Y)$ such that composition with $f$ defines a dominant rational map $f_{*}: N \dashrightarrow M$.  Then we say that $M$ is an accumulating component.  If $M$ is not an accumulating component, we call it a Manin component.
\end{defi}

We would like to understand how to differentiate accumulating and Manin components using geometric properties of the families.  

\begin{defi} \label{defi:goodcomponent}
Let $X$ be a smooth Fano threefold such that $-K_X$ is very ample.  We say that a component $M$ of $\overline{\Rat}(X)$ is a good component if it satisfies the following conditions:
\begin{enumerate}
\item $M$ parametrizes a dominant family of rational curves $C$ of anticanonical degree $\geq 3$,
\item the general map parametrized by $M$ is birational onto its image, and
\item there is no morphism $g: X \to W$ with $\dim(W) \geq 1$ such that the curves parametrized by $M$ are contracted by $g$.
\end{enumerate}
\end{defi}

\begin{theo} \label{theo:classifymanincomponents}
Let $X$ be a smooth Fano threefold such that $-K_X$ is very ample. Every good component $M$ of $\overline{\Rat}(X)$ is a Manin component.
\end{theo}

\begin{proof}
Suppose that $M$ is an accumulating component.  According to Definition \ref{defi:manincomponent} there are three possible types for the morphism $f: Y \to X$ which verifies the accumulating condition.
\begin{enumerate}
\item Suppose $f: Y \to X$ satisfies Definition \ref{defi:manincomponent}.(1).  Theorem \ref{theo:ainvanddeformations} shows that this case will occur if and only if $M$ does not parametrize a dominant family.
\item Suppose $f: Y \to X$ satisfies Definition \ref{defi:manincomponent}.(2).  In this case the map $f: Y \to X$ is an $a$-cover which has Iitaka dimension $> 0$.  Let $C$ denote a general curve parametrized by $N$.  Since $\dim(N) \geq \dim(M)$ and both families of curves are dominant, by comparing the expected dimensions of deformation we see that
\begin{equation*}
-K_{Y} \cdot C \geq -K_{X} \cdot f_{*}(C).
\end{equation*}
Rearranging, we obtain $(K_{Y} - f^{*}K_{X}) \cdot C \leq 0$.  However, $K_{Y} - f^{*}K_{X}$ is pseudo-effective and $C$ deforms in a dominant family so that we must have equality $(K_{Y} - f^{*}K_{X}) \cdot C = 0$.  In particular $C$ is contracted by the Iitaka fibration for $K_{Y} - f^{*}K_{X}$.

If the Iitaka dimension is $2$, then by Lemma \ref{lemm:iitakadim2case} $N$ parametrizes a family of curves whose set-theoretic images in $X$ are $-K_{X}$-conics.
Thus either:
\begin{enumerate}
\item the curves parametrized by $M$ have anticanonical degree $2$ in $X$, or
\item the maps parametrized by $M$ are multiple covers.
\end{enumerate}
If the Iitaka dimension is $1$, then by Theorem \ref{theo:iitakadim1case} $N$ parametrizes curves whose images in $X$ are contained in the fibers of a del Pezzo fibration on $X$.

\item Suppose $f: Y \to X$ satisfies Definition \ref{defi:manincomponent}.(3).  In this case $K_{Y} - a(Y,-f^{*}K_{X})K_{X}$ has Iitaka dimension $0$.  By Theorem \ref{theo:iitakadim0case} this cannot happen for Fano threefolds.
\end{enumerate}
We conclude that every accumulating component must fail one of the three criteria in the definition of a good component.
\end{proof}

\begin{rema} \label{rema:additionalclassifymanincomponents}
The converse of Theorem \ref{theo:classifymanincomponents} is false: not every Manin component is good.  For example, multiple covers of conics can sometimes be Manin components when $\rho(X)>1$.  Similarly, curves contracted by a del Pezzo fibration may or may not be Manin components depending on the comparison between $\rho(X)$ and $\rho(F)$ for a general fiber $F$ of the fibration.

In Theorem \ref{theo:pathologicalisnegligible} we will show that these ``pathological'' Manin components make a negligible contribution to the asymptotic growth of the counting function.  Thus to analyze Geometric Manin's Conjecture it suffices to focus on the good components.
\end{rema}

\subsection{Counting function} \label{subsec:countingfunction}

In order to emphasize the analogies between Geometric Manin's Conjecture and Manin's Conjecture, it is helpful to introduce the following counting function.  Let $X$ be a Fano variety and let $r(X,-K_{X})$ denote the minimal positive integer of the form $K_{X} \cdot \alpha$ for some $\alpha \in N_{1}(X)_{\mathbb{Z}}$.  Fix a positive constant $q>1$.  We define
\begin{equation*}
N(X, -K_{X}, q, d) = \sum_{i=1}^{d} \sum_{W \in \mathrm{Manin}_{i}} q^{\dim W}
\end{equation*}
where $\mathrm{Manin}_{i}$ denotes the set of Manin components which parametrize curves of anticanonical degree $i \cdot r(X,-K_{X})$.  Note that this counting function encodes both the dimensions and the number of components of $\mathrm{Manin}_{i}$; the dimension will determine the dominant term in the asymptotic growth of $N(X,-K_{X},q,d)$ and the number of components will determine the subdominant term. Geometric Manin's Conjecture predicts that
\begin{equation*}
N(X,-K_{X},q,d) \sim_{d \to \infty} \frac{ q^{\dim(X) - 3} \alpha(X,-K_X)}{1-q^{-r(X, -K_X)}} q^{dr(X, -K_X)} d^{\rho(X)-1}
\end{equation*}
where $\alpha(X,-K_X) = \rho(X) \cdot \mathrm{Vol}_{\mu}(\Nef_{1}(X) \cap \{ \alpha \in N_{1}(X) | -K_{X} \cdot \alpha \leq r(X,-K_X) \})$ and $\mu$ is the Lebesgue measure normalized with respect to the lattice $N_{1}(X)_{\mathbb{Z}}$. 

For a Fano threefold $X$ the counting function takes a particularly easy form.  Each Manin component will have the expected dimension $-K_{X} \cdot C$.  Thus, we can compute $N(X,-K_{X},q,d)$ by summing up the terms $C_{\alpha} q^{\ell(\alpha)}$ as we vary $\alpha \in \Nef_{1}(X)_{\mathbb{Z}}$ over all lattice points which have anticanonical degree between $2$ and $d \cdot r(X,-K_{X})$ where $C_{\alpha}$ denotes the number of Manin components of a given class and $\ell(\alpha)$ is the linear function $-K_{X} \cdot \alpha$.  In this way we see that the behavior of the counting function is entirely controlled by the values of $C_{\alpha}$.

The following lemma gives a lower bound on the asymptotic growth of the counting function.

\begin{lemm} \label{lemm:sufficientclasses}
Let $X$ be a smooth Fano threefold such that $-K_X$ is very ample. There is a full-dimensional subcone $\mathcal{K} \subset \Nef_{1}(X)_{\mathbb{Z}}$ and an element $\tau \in \mathcal{K}_{\mathbb{Z}}$ such that for every $\alpha \in \tau + \mathcal{K}_{\mathbb{Z}}$ we have $C_{\alpha} \geq 1$.
\end{lemm}

\begin{proof}
By \cite[Theorem 1.3]{TZ14} $N_{1}(X)_{\mathbb{Z}}$ is generated by classes of rational curves.  By gluing on sufficiently many copies of a very free rational curve $C$ to these generators, smoothing, and gluing, we can in fact (after adding in the class of $C$) find a generating set consisting of very free curves.  Let $\mathcal{K}$ denote the cone generated by these very free curves.  Applying \cite[\textsection 3 Proposition 3]{Khovanskii92} we see there is some translate of $\mathcal{K}$ such that every $\mathbb{Z}$-class in the translate is represented by a family of very free curves.  By Theorem \ref{theo:classifymanincomponents} each such family will be a Manin component.
\end{proof}

\begin{rema}
In fact, using the classification of Fano threefolds one can show that that every extremal ray of $\Nef_{1}(X)$ is represented by the class of a free rational curve.  So one can take $\mathcal{K} = \Nef_{1}(X)$ in the statement of Lemma \ref{lemm:sufficientclasses}. 
\end{rema}

Using standard lattice counting techniques, Lemma \ref{lemm:sufficientclasses} shows that there is some constant $\Gamma$ such that for sufficiently large $d$
\begin{equation*}
N(X,-K_{X},q,d) \geq \Gamma q^{dr(X, -K_X)} d^{\rho(X)-1}.
\end{equation*}
We are now equipped to analyze the ``pathological'' Manin components.

\begin{theo} \label{theo:pathologicalisnegligible}
Let $X$ be a smooth Fano threefold such that $-K_X$ is very ample. The Manin components $M$ which are not good components will make a negligible contribution to the asymptotic growth of the counting function, except possibly those components contracted by a del Pezzo fibration whose fibers have degree $1$.
\end{theo}

The proof uses the results of \cite{Testa09} which count components of the moduli space of rational curves on del Pezzo surfaces.  Since \cite{Testa09} does not address arbitrary del Pezzo surfaces of degree $1$, we must include these as a possible exception.

\begin{proof}
It suffices to consider only the Manin components of anticanonical degree $\geq 3$.  Definition \ref{defi:goodcomponent} identifies three defining properties of good components.  Theorem \ref{theo:ainvanddeformations} shows that every Manin component parametrizes a dominant family of curves, so it is not possible for a Manin component to fail (1).

Suppose that a Manin component $M$ fails condition (2).  A dimension count shows that $M$ must parametrize multiple covers of a family $M'$ of $-K_{X}$-conics.  Note that for any degree $d$ there are only finitely many families of multiple covers of conics which have degree $d$.  Lemma \ref{lemm:sufficientclasses} shows that such contributions will have a negligible contribution to the counting function if $\rho(X) > 1$.

If $\rho(X) = 1$, let $Y$ be a resolution of the one-pointed family of conics over $M'$ equipped with the evaluation map $f: Y \to X$.  After perhaps taking a base change over a suitable morphism $T \to M'$, we may ensure that $K_{Y} - a(Y,-f^{*}K_{X})f^{*}K_{X}$ has Iitaka dimension $2$.  By \cite[Corollary 4.16]{LST18} we have $a(Y,-f^{*}K_{X}) = 1$ and $b(Y,-f^{*}K_{X}) = 1$.  This implies that $Y$ satisfies condition (2) of Definition \ref{defi:manincomponent}.  Furthermore composition with $f$ will map the parameter space of multiple covers of conics on $Y$ dominantly onto $M$.  Thus when $\rho(X) = 1$ every component $M$ parametrizing multiple covers of conics is an accumulating component.

Finally, suppose that a Manin component $M$ fails condition (3) but not condition (2).  Then $M$ will parametrize a dominant family of curves which are contained in the fibers of a del Pezzo fibration $g: X \to \bP^{1}$.  To show that such contributions are negligible, we must calculate how many Manin components represent the classes of curves contracted by $f$.

Let $T \to \mathbb{P}^{1}$ be a base change which kills the monodromy action on the N\'eron-Severi space of the general fiber and let $Y$ be a resolution of $T \times_{\mathbb{P}^{1}} X$.  After perhaps replacing $T$ by a finite cover we may suppose that the Iitaka dimension of $K_{Y} - f^{*}K_{X}$ is $1$.  There is a bijection between the families of free curves on $Y$ contracted by the map $Y \to T$ and the families of free curves on a general fiber $F$ of $Y \to T$.  Since $F$ is a del Pezzo surface of degree $\geq 2$, by \cite{Testa09} there is a constant $Q$ such that any nef numerical class on $F$ is represented by at most $Q$ families of free curves.  By pushing forward, we see that the asymptotic contribution of curves contracted by $f$ to the counting function on $X$ is at most $Q' q^{d} d^{\rho(F) - 1}$ for some constant $Q'$.  If $\rho(F) < \rho(X)$ then by Lemma \ref{lemm:sufficientclasses} the contribution to the counting function is negligible.  If $\rho(F) \geq \rho(X)$ then by \cite[Corollary 4.16]{LST18} $f: Y \to X$ satisfies condition (2) of Definition \ref{defi:manincomponent} showing that every family contracted by $g$ is an accumulating component.
\end{proof}

\subsection{Upper bounds}
To give an upper bound on the asymptotic growth of the counting function we must bound the values of the function $C_{\alpha}$ counting Manin components representing $\alpha$.  The following result proves the conjectural bound occurring in Batyrev's heuristic.

\begin{theo} \label{theo:batyrevsconj}
Let $X$ be a smooth Fano threefold.  There is a polynomial $P(d)$ such that the number of components of $\overline{\Rat}(X)$ parametrizing curves of anticanonical degree $\leq d$ is bounded above by $P(d)$.
\end{theo}

\begin{proof}
We separate the components of $\overline{\Rat}(X)$ into two types: the non-dominant families and the dominant families.  The existence of a polynomial bound on the number of dominant families follows from Movable Bend-and-Break combined with \cite[Theorem 5.13]{LTCompos}.

By Theorem \ref{theo:maintheorem1} there is a finite set of surfaces $\{ Y_{i} \}$ in $X$ which contain all non-dominant families of rational curves on $X$, and each such family will sweep out one of these surfaces.  Let $S_{i}$ denote the resolution of $Y_{i}$.  It suffices to verify that for each $S_{i}$ there is a polynomial $P_{i}(d)$ bounding the number of dominant components of $\Rat(S_{i})$ which have degree at most $d$ against the big and nef divisor $-K_{X}|_{S_{i}}$.

The possibilities for $S_{i}$ are recorded in Theorem \ref{theo:higherainv}.  When $S_{i}$ corresponds to a contractible divisor as in Theorem \ref{theo:higherainv}.(2), by examining the finite list of possibilities we see that the space of rational curves parametrizing any numerical class on $S_{i}$ is irreducible and the existence of the desired polynomial $P_{i}(d)$ follows.  The other possibility is that $S_{i}$ maps to a family of lines on $X$ as in Theorem \ref{theo:higherainv}.(1).  In this case by Lemma \ref{lemm:linesfromruled} we may suppose that $S_{i}$ is a ruled surface. 
If $S_{i}$ is ruled over a curve of genus $\geq 1$, then the only rational curves on $S_{i}$ are the fibers of the $\mathbb{P}^{1}$-bundle structure.  If $S_{i}$ is a Hirzebruch surface, then in particular it is a toric variety so by \cite[Theorem 1.10]{Bourqui16} (plus an easy argument dealing with classes on the nef boundary) there is at most one family of rational curves representing any numerical class on $S_{i}$.  Thus the number of components with a fixed intersection against $-K_{X}|_{S_{i}}$ grows at most linearly in the degree.
\end{proof}

 Theorem \ref{theo:batyrevsconj} implies that we obtain an upper bound on the counting function of the expected form
\begin{equation*}
N(X,-K_{X},q,d) = O(q^{dr(X, -K_X)}d^{s})
\end{equation*}
for some positive integer $s$.  Conjecturally the optimal bound is $s = \rho(X)-1$.  The following conjecture would imply that the counting function has the conjectural asymptotic growth rate and has the expected leading constant.

\begin{conj} \label{conj:uniquecomponent}
Let $X$ be a smooth Fano threefold such that $-K_X$ is very ample. There is some $\tau \in \Nef_{1}(X)_{\mathbb{Z}}$ such that for every $\alpha \in \tau + \Nef_{1}(X)_{\mathbb{Z}}$ there is at most one good component representing $\alpha$.
\end{conj}

\begin{rema}
In fact, we do not know of \emph{any} curve class which lies in the interior of the nef cone of a Fano threefold and is represented by more than one good component of $\overline{\Rat}(X)$.

Note however that it is not true that every curve class in the interior of the nef cone of a Fano threefold is represented by at most one family of free rational curves.  \cite[Proposition 2.1.2]{Iliev94} shows that the general Gushel threefold (i.e.~a double cover of a codimension $3$ linear section of $G(2,5)$ branched over the intersection with a quadric hypersurface) carries two different families of anticanonical conics.
\end{rema}

\section{Examples} \label{sec:examples}

As explained in the introduction, in principle one can use our main theorems to classify all families of rational curves on a Fano threefold.  In this section we compute a couple of examples. 

\subsection{Quartic threefolds} \label{sect:quarticthreefold}
We classify the components of the moduli space of rational curves on every smooth quartic threefold $X$ in $\mathbb{P}^{4}$.  In particular, we show that for every $e \geq 3$ there is only one irreducible component of $\overline{M}_{0,0}(X,e)$ whose general points parametrize generically injective morphisms of degree $e$ from $\mathbb{P}^1$ to $X$. 

\begin{lemm}
\label{lemm:quartic_no_one_parameter}
Let $X$ be a smooth hypersurface of degree $4$ in $\mathbb{P}^4$ and $S \subset X$ a rational surface. There are at most finitely many lines, conics, or smooth twisted cubics on $S$. 
\end{lemm}

\begin{proof}
Assume to the contrary that there is a 1-parameter family $\mathcal C$ of smooth rational curves of degree $\leq 3$ on $S$. 
Let $\tilde{S} \to S$ be a desingularization and $D \subset \tilde{S}$ the strict transform of a general curve in this family. Suppose the space $\mathcal D$ of deformations of $D$ in $\tilde{S}$ is $k$-dimensional.  Since $D$ is a smooth rational curve, we have $D^{2} = -K_{\tilde{S}} \cdot D - 2 = k-1$.
If $Y$ is the blow up of $\tilde{S}$ at $k-1$ general points, then the strict transforms of the rational curves parametrized by $\mathcal D$ through these $k-1$ general points give a 1-parameter family of rational curves on $Y$ of self intersection $0$.   Since $Y$ is a rational surface, this gives morphisms $\pi: Y \to \mathbb{P}^1$ and $f: Y \to X$ such that a general fiber of $\pi$ is a smooth rational curve mapped by $f$ isomorphically onto a curve of degree $\leq 3$ in $X$. Since there is an injective map $T_Y \to f^*T_X$, we conclude that 
$\wedge^2 f^*T_X \otimes \omega_Y$ has a non-zero section. 
This contradicts \cite[Proposition 2.4]{b-s}.
\end{proof}

We will say that a rank $2$ vector bundle $\mathcal{O}(a) \oplus \mathcal{O}(b)$ on $\mathbb{P}^{1}$ is balanced if $|b-a| \leq 1$.

\begin{lemm}\label{balanced}
Let $X$ be a smooth hypersurface of degree $4$ in $ \mathbb{P}^4$ and let $M$ be an irreducible component of $\overline{M}_{0,0}(X,e)$ parametrizing a dominant family of stable maps whose general members are generically injective and have irreducible domains. If $e \geq 2$ then $N_{f/X}$ is balanced for a general map parametrized by $M$.  If $e\geq 4$ then a general map parametrized by $M$ is very free. 
\end{lemm}

\begin{proof}
The arguments of \cite[Lemma 8.1]{LTJAG} combined with Lemma~\ref{lemm:quartic_no_one_parameter} show that if $e \geq 5$ then a general map must be very free.  The balanced property for the normal bundle follows from \cite{Shen12}.  So we only need to consider $2 \leq e \leq 4$.

The statement is trivial for $e=2,3$, so we assume $e =4$. Suppose $N_{f/X} = \mathcal O \oplus \mathcal O(2)$. Consider a general point $p$ in $X$. Then the family of stable maps parametrized by $M$ whose images pass through $p$ is of dimension $2$ and it sweeps out a surface $S$ in $X$. 
Since $S$ contains a $2$-parameter family of irreducible rational curves, it is a rational surface. Let $q$ be a general point on $S$, then 
there is a $1$-parameter family of quartic rational curves through $p$ and $q$ in $S$. By Lemma \ref{bend-and-break}, this family parametrizes a reducible curve with two distinct components through $p$ and $q$. Since $p$ is a general point of $X$, there is no line on $X$ through $p$, and since there are finitely many lines on $S$, there is 
no line on $S$ through $q$. So the reducible curve consists of two conics, and therefore there is a $1$-parameter family of conics on $S$ contradicting Lemma \ref{lemm:quartic_no_one_parameter}.
 \end{proof}

For a smooth hypersurface $X$ in $\mathbb{P}^4$ and for $e \geq 1, k \geq 0$, let $N_{k}(X,e) \subset \overline{M}_{0,k}(X,e)$ be the locus of free stable maps with $k$ marked points such that the domain is irreducible and the map is generically injective.  Denote by $\overline{N}_{k}(X,e)$ the closure of $N_{k}(X,e)$ in  $\overline{M}_{0,k}(X,e)$ and by 
$\alpha_{k,e}: \overline{N}_{k}(X,e) \to X^{\times k}$ the evaluation map. 
For any $2$-plane $H$ in $\mathbb{P}^4$, let $\overline{N}_{k}(X,e,H)$ denote the divisor in $\overline{N}_{k}(X,e)$ parametrizing stable maps whose image intersect $H$
and let $\alpha_{k,e,H}$ be the restriction of the evaluation map to  $\overline{N}_{k}(X,e,H)$.

\begin{lemm}\label{glp}
Fix $k \leq n$, and let $p_1, \dots, p_{k}$ be $k$ points in general linear position in $\mathbb{P}^n$.  If the evaluation map 
$ev: \overline{M}_{0,k}(\mathbb{P}^n, e) \to (\mathbb{P}^n)^{ \times k}$ is dominant, then the fiber over $(p_1, \dots, p_{k})$ is non-empty and irreducible.
\end{lemm}

\begin{proof}
Since automorphisms of $\mathbb{P}^n$ act transitively on the set of $k$-tuples of points in general linear position, the fiber over $(p_{1},\ldots,p_{k})$ is non-empty.  Note that since the moduli stack $\overline{\mathcal M}_{0,k}(\mathbb{P}^n,e)$ is smooth, general fibers of the evaluation map $\overline{\mathcal M}_{0,k}(\mathbb{P}^n, e) \to (\mathbb{P}^n)^{ \times k}$ are smooth, hence it is enough to show they are connected. 
Let $\overline{M}_{0,k}(\mathbb{P}^n, e) \to Y \to (\mathbb{P}^{n})^{\times k}$ be the Stein factorization. Since $\overline{M}_{0,k}(\mathbb{P}^n,e)$ is normal, $Y$ is normal, and therefore the branch locus of $Y \to \mathbb{P}^{n}$ is either empty or of codimension 1.  Since automorphisms of $\mathbb{P}^n$ act transitively on the set of $k$-tuples of points in general linear position and since the locus of $k$-tuples which are not in general linear position is of codimension greater than 1, the branch locus must be empty. 
\end{proof}

\begin{lemm}
Suppose $X$ is a general hypersurface of degree $4$ in $\mathbb{P}^4$ and $H$ a general $2$-plane. 
\begin{itemize}
\item [(i)] General fibers of $\alpha_{1,3}: \overline{N}_{1}(X,3) \to X$ and $\alpha_{2,5}: \overline{N}_2(X,5) \to X \times X$ are irreducible curves. 
\item [(ii)] General fibers of $\alpha_{1,4,H}$ and $\alpha_{2,6,H}$ are irreducible curves. 
\end{itemize}
\end{lemm}

\begin{proof}
We first prove (i).  Note that the space of singular irreducible rational cubics on a general quartic hypersurface in $\mathbb{P}^4$ is 1-dimensional, so there is no such cubic through a general point of $X$. Therefore it is enough to show the space of twisted cubics through a general point of $X$ is irreducible. Fix a point $p$ in $\mathbb{P}^4$, and denote by $T$ the space of all smooth twisted cubics in $\mathbb{P}^4$ through $p$. Let $\pi: \mathcal C \to T$ be the universal curve over $T$ and $\alpha: \mathcal C \to \mathbb{P}^4$ the universal map. Let $D= \alpha^{-1}(p)$. Then $E := \pi_* (\alpha^* \mathcal O(4) \otimes I_{D/\mathcal C})$ is a locally free sheaf on $T$, and for any hypersurface $X=\{f=0\}$ containing $p$, 
$\sigma : = \pi_* \alpha^* f$ is a section of $E$ whose zero locus is the space of twisted cubics through $p$ on $X$. Since any twisted cubic $C$ is linearly normal, 
$H^0(\mathcal O(4)) \to H^0(\mathcal O(4)|_C)$ is surjective and hence $E$ is globally generated. The fibers of $\alpha$ are connected by Lemma \ref{glp}, so we can identify sections of $E$ with sections of $\mathcal O_{\mathbb{P}^4}(4) \otimes I_{p/\mathbb{P}^4}$. The desired result now follows since the zero locus of a general section of a globally generated locally free sheaf on $T$ is irreducible. 

The proof for $\alpha_{2,5}$ is similar if we fix two general points $p_1$ and $p_2$ in $\mathbb{P}^4$ and consider the space of smooth non-degenerate quintic rational 
curves through $p_1$ and $p_2$. If $E=\pi_* (\alpha^*\mathcal O(4) \otimes I_{\alpha^{-1}(p_1\cup p_2)/\mathcal C})$, then since any non-degenerate quintic $C$ in $\mathbb{P}^4$ is 2-normal,  $E$ is globally generated.
Note that the space of degenerate quintic rational curves 
on a general $X$ is of dimension 3 and therefore there is no such quintic through two general points of $X$.

We next prove (ii).  A similar argument as in part (i) shows that general fibers of $\alpha_{1,4}$ are irreducible surfaces, and since $H$ is general, the result follows. Similarly for $\alpha_{2,6}$, if $C$ is a smooth non-degenerate degree 6 rational curve in $\mathbb{P}^4$, then $C$ is $3$-normal. So $E$ is globally generated, and 
the zero locus of a general section is an irreducible surface. 
\end{proof} 

\begin{coro}\label{irr-rat}
Suppose that $X$ is an arbitrary smooth hypersurface of degree $4$ in $\mathbb{P}^4$. Then all fibers of 
$\alpha_{1,3}$, $\alpha_{2,5}, \alpha_{1,4,H},$ and $\alpha_{2,6,H}$ are connected curves. 
\end{coro}

\begin{proof}
 Denote by $P^0$ the space parametrizing smooth hypersurfaces of degree $4$ in $\mathbb{P}^4$  and consider the incidence correspondence 
$\mathcal I = ([X], (f, q_1, \dots, q_k)) \subset P^0 \times \overline{M}_{0,k}(\mathbb{P}^4,e)$ where $f: \mathbb{P}^1 \to X$ is a generically injective map.   If $e\leq 6$, the fibers of the projection of $\mathcal I$ to $\overline{M}_{0,k}(\mathbb{P}^4,e)$ are irreducible of codimension $4e+1$ in $P^0$  since any curve of degree $\leq 5$ and any non-planar curve of degree $6$ in $\mathbb{P}^4$ is 4-normal.  Therefore $\mathcal I$ is irreducible. Denote by $\overline{\mathcal I}$ the closure of $\mathcal I$ in 
$P^0 \times \overline{M}_{0,k}(\mathbb{P}^4, e)$. Denote by $P^0_1$ and $P^0_2$ the space of smooth hypsersurfaces of degree $4$ with one and two marked points, respectively. By the previous lemma when $(k,e)=(1,3)$, general fibers of the projection map $ \overline{\mathcal I} \to P^0_1$ sending $([X], f, q_1)$ to $([X],f(q_1))$ are connected, and when $(k,e)=(2,5)$ 
the general fibers of $\overline{\mathcal I} \to P^0_2$ sending $([X], f, q_1, q_2)$ to $([X], f(q_1),f(q_2))$ are connected. Since $\overline{\mathcal I}$ is irreducible in these two cases, all the fibers should be connected.  A similar argument 
shows the statement for $\alpha_{1,4,H}$ and $\alpha_{2,6,H}$. 
\end{proof}

\begin{theo} \label{theo:veryfreecurvesonquartic}
For any smooth hypersurface of degree $4$ in $\mathbb{P}^4$, the moduli space of generically injective stable maps $\mathbb{P}^1 \to X$ onto free curves of degree $e$ is irreducible for $e \geq 3$. 
\end{theo}

\begin{proof}
First we prove the statement for $3 \leq e \leq 6$. Let $M_1, \dots, M_k$ be all the irreducible components of 
$\overline{M}_{0,0}(X,e)$ whose general points  parametrize generically injective stable maps from $\mathbb{P}^1$ to $X$ whose image is free, and suppose to the contrary that $k \geq 2$.

\begin{itemize}

\item $e=3$: Consider a general point $p$ in $X$ and the curves in each $M_i$  parametrizing stable maps through $p$. By Corollary \ref{irr-rat} the space of such curves is connected, so there should be a singular point $[f]$ of  $\overline{M}_{0,0}(X,3)$ whose image passes through $p$. The domain of $f$ cannot be irreducible since any irreducible rational curve through a general point is free. The only other possibility is that the domain has two irreducible components, one mapped isomorphically to a line $L$ and one mapped isomorphically to a conic $C$ through $p$ intersecting $L$. The normal bundle of  a free conic in $X$ is $\mathcal O \oplus \mathcal O$, and the normal bundle of $L$ in $X$ is either $\mathcal O \oplus \mathcal O(-1)$ or $\mathcal O(1) \oplus \mathcal O(-2)$. Varying $p$ we get a $2$-parameter family of conics and a $1$-parametric family of lines. Let $Y \subset X$ be the surface swept out by such lines. Then by Proposition \ref{prop:transverseintersection}, for a general $C$ in a 2-dimensional family of conics, $C$ 
intersects $Y$ transversally. So by Proposition \ref{prop:GHS},  $N_f|_C = \mathcal O(1) \oplus \mathcal O$, and $N_{f}|_L = \mathcal O\oplus \mathcal O$ or $\mathcal O(1) \oplus \mathcal O(-1)$. Therefore $[f]$ is a smooth point of $\overline{M}_{0,0}(X,3)$, a contradiction.

\item $e=4$: Let $p$ be a general point of $X$ and $H$ a general 2-plane in $\mathbb{P}^4$. Consider the curves in each $M_i$ parametrizing stable maps whose images intersect $H$ and 
pass through $p$. By Corollary \ref{irr-rat} the space of such curves is a connected curve, so there is a map $f$ parametrized by this curve which is a singular point of $\overline{M}_{0,0}(X,4)$. Since any irreducible rational curve through a general point is free, there are only two possibilities: \begin{enumerate}
\item the domain of $f$ has one component mapped birationally onto a cubic $C$ through $p$ and intersecting $H$ and one component mapped to a line $L$ on $X$, or
\item the domain of $f$ has two components, one mapped to a free conic $C_1$ through $p$ and one mapped to a non-free conic $C_2$ intersecting $H$.
\end{enumerate}
In the first case, arguing as for $e=3$ we see that $[f]$ is a smooth point of the moduli space. 
In the second case, $C_2$ should belong to a $2$-parameter family of non-free conics since it intersects both a general $2$-plane and a conic through a general point. 
By Lemma \ref{lemm:quartic_no_one_parameter} there is no $2$-parameter family of non-free conics on $X$ so $C_2$ should be free and therefore $f$ is a smooth point of the moduli space. 

\item $e=5$: The proof is similar to the case $e=3$ by looking at quintics through 2 general points.
 Note that by Lemma \ref{balanced}, the locus of stable maps through two general points in every $M_i$ is 1-dimensional. 

\item $e=6$: The proof is similar to the case $e=4$ by looking at degree 6 rational curves through two general points intersecting a general 2-plane in $\mathbb{P}^4$. By Corollary~\ref{irr-rat}, the locus parametrizing stable maps through two general points in every $M_i$ is 2-dimensional, so there is a connected curve parametrizing stable maps through two general points and intersecting a general 2-plane. Since we assume $\overline{M}_{0,0}(X,e)$ has at least two irreducible components there is a map $f$ parametrized by this curve such that $[f]$ is a singular point of the moduli space, so the restriction of $f$ to the irreducible components of its domain cannot all be free. The only possibilities are:
\begin{enumerate}
\item the domain of $f$ has two components, one mapped to a quintic and one to a line intersecting the quintic, and
\item the domain of $f$ has two components, one mapped to a quartic through two general points and one to a conic intersecting both the quartic and a general $2$-plane.
\end{enumerate}
In the first case, the line has to be general in its deformation class so the same argument as in cases $e=3, 4$ gives a contradiction. In the second case, the conic has to be general in its deformation class, so it has to be free and again we get a contradiction. 
\end{itemize}

Now we prove the statement for $e \geq 7$ using induction.
Fix any component $M \subset \overline{N}_{0}(X, e)$. By definition $M$ generically parametrizes a dominant family of free curves, and in particular it has the expected dimension.  

Next we claim that $M$ contains a stable map whose domain consists of two irreducible curves, one mapping to a degree $3$ free curve and another mapping to a degree $e-3$ free curve. Indeed, using Movable Bend-and-Break we see that $M$ contains a chain of free curves where every component has degree $\leq 4$.  Using the irreducibility of $\overline{N}_{0}(X, 4)$, we can further break each quartic into a chain of two conics. 
By smoothing a subcurve consisting of all but one component of the chain of free curves, we see that $M$ contains a stable map consisting of either (i) a degree $3$ free curve and a degree $e-3$ free curve or (ii) a free conic and a degree $e-2$ free curve.  In the latter case, since $e-2 \geq 5$, by induction on the degree we can break a degree $e-2$ free curve into the union of a free cubic and a degree $e-5$ free curve. Then using \cite[Lemma 5.11]{LTCompos} we may change the order of three free curves while remaining in $M$ so we may assume that $M$ contains the union of a free conic, a degree $e-5$ free curve, and a free cubic such that a degree $e-5$ free curve is attached to a conic and a cubic. Then we smooth the conic and the degree $e-5$ free curve to deduce our assertion.

By Theorem \ref{theo:maintheorem2} the evaluation map for any family of free rational curves of degree $\geq 3$ will have connected fibers.  Thus by the inductive assumption there is a unique main component of 
\[
N_1(X, 3) \times_X N_{1}(X, e-3)
\]
and this locus is contained in $M$.  Because the union of two free curves is a smooth point of $\overline{M}_{0,0}(X, e)$, we deduce that there is a unique component $M$. 
\end{proof}

Finally, we can complete the classification of components $\overline{\Rat}(X)$ of anticanonical degree $\geq 3$.  Theorem \ref{theo:veryfreecurvesonquartic} handles all dominant families which are generically birational.  An easy dimension count shows that any dominant family which is generically non-birational must parametrize multiple covers of a dominant family of conics.  Finally, suppose we have a non-dominant family of rational curves on $X$.  Theorem \ref{theo:higherainv} shows that this family is contained in a surface $Y$ in $X$ swept out by lines.  Furthermore Lemma \ref{lemm:quartic_no_one_parameter} shows that $Y$ cannot be rational.  Thus the only rational curves in $Y$ are the lines in $Y$ and a general member of our family must define a multiple cover of a line.

\subsection{The Fano threefold with two E5 contractions} \label{sect:2E5contractions}
Let $X$ be the blow-up of $\mathbb{P}_{\mathbb{P}^{2}}(\mathcal{O} \oplus \mathcal{O}(2))$ along a quartic curve in a minimal moving section.  

\begin{theo}
\label{theo:classificationTwoE5}
There is a unique family of free curves representing each nef curve class $\alpha$ on $X$. For any pseudo-effective curve class $\alpha$ on $X$ there is at most one component of $\overline{\Rat}(X)$ representing this class.
\end{theo}

The rest of this section will be devoted to the proof of Theorem \ref{theo:classificationTwoE5}. We let $\phi: X \to \mathbb{P}_{\mathbb{P}^{2}}(\mathcal{O} \oplus \mathcal{O}(2))$ denote the blow-up of a quartic curve $Z$ contained in a minimal moving section $D$ of the projective bundle.  Let $E$ denote the exceptional divisor for $\phi$, let $E_{0}$ denote the strict transform of the rigid section of the projective bundle, and let $E_{\infty}$ denote the strict transform of $D$.  Let $p: X \to \mathbb{P}^{2}$ denote the composition of $\phi$ with the projective bundle map and let $H$ denote the pullback of the hyperplane class under $p$.  The two E5 divisors are $E_{0}$ and $E_{\infty}$.

We let $E'$ denote the component of the $p$-preimage of the quartic curve $p(E)$ which is different from $E$.  Note that $E + E' \sim 4H$.  It turns out that $E'$ is a contractible divisor and contracting it yields the variety $\mathbb{P}_{\mathbb{P}^{2}}(\mathcal{O} \oplus \mathcal{O}(2))$ (in a different way from $\phi$).

The divisors $E_{0}, E_{\infty}, E, E'$ generate the pseudo-effective cone of divisors.  Dually, the nef cone of curves on $X$ has 4 extremal rays.  These extremal rays are generated by the following curve classes: 
\begin{itemize}
\item[] $R_{1}$: the class of a general fiber of $p$.  
\item[] $R_{2}$: the class of the $\phi$-strict transform of a line in a minimal moving section of $\mathbb{P}_{\mathbb{P}^{2}}(\mathcal{O} \oplus \mathcal{O}(2)) \to \mathbb{P}^{2}$ that meets the quartic curve $Z$ in two points. 
\item[] $R_{3}$: the class of the $\phi$-strict transform of a general line in a minimal moving section of $\mathbb{P}_{\mathbb{P}^{2}}(\mathcal{O} \oplus \mathcal{O}(2)) \to \mathbb{P}^{2}$. 
\item[] $R_{4}$: the class of the construction analogous to $R_{3}$ that uses the other birational contraction to $\mathbb{P}_{\mathbb{P}^{2}}(\mathcal{O} \oplus \mathcal{O}(2))$.  
\end{itemize}
These together with the following curve classes
\begin{itemize}
\item[] $R_{5}$: the class of the $\phi$-strict transform of a line in a minimal moving section of $\mathbb{P}_{\mathbb{P}^{2}}(\mathcal{O} \oplus \mathcal{O}(2)) \to \mathbb{P}^{2}$ that meets the quartic curve $Z$ in one point. 
\item[] $R_{6}$: the class of the construction analogous to $R_{5}$ that uses the other birational contraction to $\mathbb{P}_{\mathbb{P}^{2}}(\mathcal{O} \oplus \mathcal{O}(2))$. 
\end{itemize}
generate the monoid $\Nef_{1}(X)_{\mathbb{Z}}$.  Besides the $R_{i}$ there is exactly one other nef curve class which has anticanonical degree $\leq 5$, namely, $R_{1}+R_{2}$.

Let $\mathcal{R}$ denote the commutative monoid generated freely by the formal symbols $R_{i}$.  We have a natural numerical class homomorphism $\mathcal{R} \to \Nef_{1}(X)_{\mathbb{Z}}$, and wish to understand precisely when two elements of $\mathcal{R}$ are identified by this map. 

\begin{prop}
Two elements in $\mathcal{R}$ are identified by the homomorphism $\mathcal{R} \to \Nef_{1}(X)_{\mathbb{Z}}$ if and only if they can be identified using a sequence of the following relations while remaining in $\mathcal{R}$:
\begin{align*}
R_{2} + R_{3} = 2R_{5} \qquad R_{2} + R_{4} & = 2R_{6} \qquad R_{1} + 2R_{2} = R_{5} + R_{6} \\
R_{1} + R_{2} + R_{5} = R_{3} + R_{6}&  \qquad R_{1} + R_{2} + R_{6} = R_{4} + R_{5} \\
R_{1} + R_{5} + R_{6} & = R_{3} + R_{4}
\end{align*}
\end{prop}

\begin{proof}
To see this, let $$\sum_{i=1}^6 a_i R_i = \sum_{i=1}^{6} b_iR_i$$ be a relation between the $R_i$ in $\Nef_{1}(X)$ with $a_i, b_i$ non-negative integers. We show by induction on the anticanonical degree $m$ that the right hand side can be obtained from the left hand side after a finite number of substitutions using the above relations. If $m=0$, there is nothing to prove. Suppose we know the statement for every positive integer up to $m-1$.  If for some $i$ both $a_i$ and $b_i$ are positive we can use our induction hypothesis to show the statement, so we may assume no $R_i$ appears on both sides of the equation. Suppose $a_3 >0$ and $b_3=0$. Then at least one of $a_2,a_4,a_6$ should be nonzero. The reason is that $R_3$ is the only class among the $R_i, 1 \leq i \leq 6$ with negative intersection with $H+E_0-E_\infty$, the intersection of both $R_1$ and $R_5$ with $H+E_0-E_\infty$ is zero, and the intersection of $R_2,R_4,R_6$ with $H+E_0-E_{\infty}$ is positive. Using the above $6$ relations, we can replace any of $R_3+R_2$, $R_3+R_4$, $R_3+R_6$ with a combination which does not involve $R_3$.  Repeating the argument, we get to a relation which does not involve $R_3$.  (Note that this operation does not change the value of our induction variable $m$.)

Now if $R_4$ appears on one side of the equation at least one of $R_2$ or $R_5$ should appear on that side because $R_4$ is the only class 
with negative intersection with $H+E_{\infty}-E_0$ and $R_1$ and $R_6$ have intersection $0$ with it.  Just as before, we can use the above relations to reduce to the case when $a_4=b_4=0$. So we are left with a relation involving $R_1, R_2, R_5$ and $R_6$.  In this case we consider the intersection with $E_0-E_{\infty}$ to conclude that $R_5$ and $R_6$ should both appear on one side of the equation with the same multiplicity since the intersection of $R_1$ and $R_2$ with $E_0-E_\infty$ is zero, $R_5$ has intersection $-1$ with $E_0-E_\infty$ and $R_6$ has  intersection $1$ with it. So by replacing $R_5+R_6$ by $R_1+2R_2$ we get the desired result.
\end{proof}

We will classify the families of curves on $X$ in a sequence of four claims:
\begin{enumerate}
\item We first classify non-dominant families of rational curves.  In particular, we will show that there is at most one such family representing any pseudo-effective curve class and that such families never have nef numerical classes.
\item We then prove there is a unique family of free curves representing any nef numerical class of anticanonical degree at most $9$, and for one particular nef class of anticanonical degree $10$.
\item We prove Movable Bend-and-Break for families of free curves of anticanonical degree at least $6$.
\item We use the two previous properties to prove there is a unique family of free curves representing any nef numerical class.
\end{enumerate}

Together (1) and (4) show Theorem \ref{theo:classificationTwoE5}. 

\begin{clai}
Let $X$ be as above.  Then any non-dominant family of rational curves will sweep out either $E_{0}$, $E_{\infty}$, $E$, or $E'$.  Every numerical class is represented by at most one non-dominant families and there is no non-dominant family with nef numerical class.
\end{clai}

We start by classifying $-K_{X}$-lines.  Since $-K_{X} = 3H + E_{0} + E_{\infty}$, every $-K_{X}$-line will satisfy one of the following:
\begin{enumerate}
\item $H \cdot \ell = 0$,
\item $E_{0} \cdot \ell < 0$, or
\item $E_{\infty} \cdot \ell < 0$.
\end{enumerate}
Thus we see that $\ell$ must be contained in one of $E, E', E_{0}, E_{\infty}$.  Furthermore, these four divisors are the only contractible divisors on $X$.  By Theorem \ref{theo:higherainv} we see that any non-dominant family of curves will sweep out one of these four divisors.

If this divisor is $E$ then the family must be a multiple cover of the unique family of lines contained in the divisor.  There is a unique such family in any numerical class.  Since these curves satisfy $E \cdot C < 0$ they are not nef and cannot have the same numerical class as a family of curves sweeping out any other divisor.  A similar argument works for $E'$.

If this divisor is $E_{0}$ then the family must be the unique family of rational curves in $\mathbb{P}^{2}$ of the appropriate degree.  There is a unique such family in any numerical class.  Since these curves satisfy $E_{0} \cdot C < 0$ they are not nef and cannot have the same numerical class as a family of curves sweeping out any other divisor.  A similar argument works for $E_{\infty}$.

\begin{clai} \label{clai:e5irreducibility}
Let $X$ be as above.  Suppose that $\alpha$ is a nef numerical class of anticanonical degree $\leq 9$ or $\alpha = R_{3} + R_{4}$ of anticanonical degree $10$.  Then there is only a single component of the parameter space of rational curves representing that class whose general element is a free rational curve.
\end{clai}

Note that $\mathbb{P}_{\mathbb{P}^{2}}(\mathcal{O} \oplus \mathcal{O}(2))$ admits a $\mathbb{C}^{*}$-action which acts on the fibers of the projective bundle and fixes the rigid section and the section $D$.  Thus we obtain a compatible $\mathbb{C}^{*}$-action on $X$ that fixes pointwise the divisors $E_{0}$ and $E_{\infty}$.  (If we instead use the other birational contraction to $\mathbb{P}_{\mathbb{P}^{2}}(\mathcal{O} \oplus \mathcal{O}(2))$ to obtain an induced $\mathbb{C}^{*}$-action the resulting action is the inverse of the first.)  Our general strategy is to use this $\bC^*$-action to show that any component must contain a particular type of curve in its smooth locus. We will then analyze curves of that type and conclude that there can only be one component.

We start by classifying the rational curves on $X$ which are free but not very free.  Let $C$ be a general member of such a family; in particular, we may ensure that $C$ avoids the locus $E \cap E'$ where $p$ fails to be smooth.  Thus along $C$ we have a relative tangent bundle sequence
\begin{equation*}
0 \to \mathcal{O}_{X}(E_{\infty} + E_{0})|_{C} \to T_{X}|_{C} \to p^{*}T_{\mathbb{P}^{2}}|_{C} \to 0
\end{equation*}
Since we are assuming $C$ is not very free, the restriction of either the first or the last term to $C$ must have a non-positive summand.  If the restriction of the last term to $C$ has a non-positive summand, then $C$ is a fiber of $p$.  If the restriction of the first term to $C$ is non-positive, then $E_{\infty} \cdot C = 0$ and $E_{0} \cdot C = 0$.  We will systematically handle curves of these types in Lemma \ref{lemm:nonveryfreecurveson2E5}.  Thus for the remainder of the argument we may assume that $C$ is very free.

Let $C$ be a general member of a family of very free rational curves on $X$.  We first claim that $C \cap (E_{\infty} \cup E)$ consists of reduced points which lie in distinct $\mathbb{C}^{*}$-orbit closures.   By \cite[II.3.7 Proposition]{Kollar} a general free curve avoids any fixed codimension $2$ locus, so we may ensure that $C$ avoids $E_{\infty} \cap E$.  Thus it suffices to show the statement separately for $C \cap E_{\infty}$ and $C \cap E$.  To see the statement for $C \cap E_{\infty}$, note that by Proposition \ref{prop:transverseintersection} $C$ is not tangent to $E_{\infty}$.  To see the statement for $C \cap E$, note that by Proposition \ref{prop:transverseintersection}, we know that $C$ is not tangent to $E$ at any point of intersection and by Proposition \ref{prop:veryfreeavoidsfibers} the intersection points lie in different fibers of the $\mathbb{P}^{1}$-bundle structure.

We now split the argument into two parts, depending upon whether or not $E' \cdot C \geq E \cdot C$.  If this inequality is satisfied, we will next act on $C$ by the $\mathbb{C}^{*}$-action in order to deform the curve into $E_{0}$.  If this equality is not satisfied, we will instead take the inverse limit to deform the curve into $E_{\infty}$.  Since the situation is entirely symmetric, it suffices to consider the case when $E' \cdot C \geq E \cdot C$.

Let $C'$ denote the limit of $C$.  Let $s = E \cdot C$ and let $t = E_{\infty} \cdot C$.  Then $C'$ will be a union of an irreducible rational curve $C_{0}$ contained in $E_{0}$, $s$ curves $S_{i}$ that are $p$-vertical with normal bundle $\mathcal{O} \oplus \mathcal{O}(-1)$, and $t$ curves $T_{j}$ that are $p$-vertical and free.  By the analysis above we know that the intersection points of the $p$-vertical curves with $C_{0}$ are distinct points.  Furthermore, Proposition \ref{prop:imageisnodal} shows that $C_{0}$ is a rational curve of degree $d := H \cdot C$ in $\mathbb{P}^{2}$ with at worst nodes.  Since we are in the case when $E' \cdot C \geq E \cdot C$ we see that $2d \geq s$.

We next verify that the limit stable map is a smooth point of $\overline{M}_{0,0}(X)$.  Let $f: W \to X$ denote the stable map whose image is $C'$.  Since $f$ is an embedding in a neighborhood of each node of $W$, we can calculate the restriction of $N_{f/X}$ to each component of $Z$ using \cite[Lemma 2.6]{GHS03}.  Using the fact that each $S_{i}, T_{j}$ intersects $E_{0}$ transversally, we see that the restriction of $N_{f/X}$ to every component of $W$ is globally generated, and we conclude that $f$ is a smooth point by Proposition \ref{prop:vanishingofH1}.

The following lemma shows that there are relatively few combinatorial types for these limit stable maps.

\begin{lemm}
\label{lemm:ratcurvesthroughpointsinp2deg11}
Choose positive integers $s$ and $d$ such that $d \leq 2$, $s \leq 4$ if $d = 2$, and $s \leq 2$ if $d = 1$.  Fix distinct points $p_1, \cdots, p_s$ of $\bP^2$.  Then the subset of $\overline{M}_{0,0}(\mathbb{P}^{2})$ parametrizing degree $d$ immersions with irreducible domains passing through $p_1, \cdots, p_s$ is irreducible and is smooth.
\end{lemm}

\begin{proof}
The smoothness part follows from a normal bundle calculation.  Indeed, for $f$ an immersion from $\bP^1$, $N_{f/\bP^2} = \cO_{\bP^1}(3d-2)$.  If we twist down by $s$ points, we still have that $N_{f/\bP^2}(-p_1-\cdots - p_s)$ has vanishing $H^1$ for $s$ and $d$ in the range described by the theorem.

For irreducibility, note that since $d=1$ or $2$ we are simply considering an open subset of the linear system of degree $d$ curves passing through $s$ points, and this is clearly irreducible.
\end{proof}

Fix a nef curve class $\alpha$ and set $d= H \cdot \alpha$, $s = E \cdot \alpha$, and $t = E_{\infty} \cdot \alpha$.  Since $E_{0} \cdot \alpha \geq 0$ we have $2d \leq s+t$.  We are also assuming that $s \leq 2d$.  Finally, we are either in the case when $d + s + 2t \leq 9$ or, if $\alpha = R_{3} + R_{4}$, when $d=t=2$ and $s=4$.  In all these cases except possibly if $\alpha$ is a multiple of $R_{2}$ (which we will handle in Lemma \ref{lemm:nonveryfreecurveson2E5}) we conclude that $s$ and $d$ satisfy the conditions of Lemma \ref{lemm:ratcurvesthroughpointsinp2deg11}. 

We would like to show that the space of free rational curves parametrizing $\alpha$ is irreducible.
Let $M$ be any component of $\overline{M}_{0,0}(X)$ which generically parametrizes free curves of class $\alpha$.  We know that $M$ contains a stable map whose image is a limit curve $f': C' \to X$ as before, and furthermore we know that each such $f'$ is a smooth point and is uniquely determined by $C'$.  Thus, it suffices to show that the parameter space $P$ of limit curves $C'$ is irreducible.  By sending a general curve to its $\mathbb{C}^{*}$-limit, the parameter space $P$ admits a rational map $\psi$ to the moduli space of tuples $(C_{0}, \{p_{i}\},\{q_{j}\})$ where $C_{0}$ is a point in the Hilbert scheme of rational degree $d$ plane curves, $\{ p_{i} \}_{i=1}^{s}$ are points which lie in $C_{0} \cap Z$ representing the attachment points with the $S_{i}$, and $\{ q_{j} \}_{j=1}^{t}$ are points which lie in $C_{0} \backslash Z$ representing attachment points with the $T_{j}$.  Note that the limit curve $C'$ is determined by $C_{0}$ and the attachment points $p_{i}$, $q_{j}$.  In fact, $\psi$ is birational, since for general choices of $C_{0}, p_{i}, q_{j}$ the corresponding stable map has globally generated normal sheaf and can thus be smoothed back into a free curve.

Next consider the forgetful map to the moduli space $Q$ of tuples $(C_{0},\{p_{i}\})$.  The fibers of this map are open subsets of $(\mathbb{P}^{1})^{\times t}$.  Thus, to show that $P$ is irreducible it suffices to prove that the moduli space $Q$ is irreducible.  Note that $Q$ admits a forgetful map to $\Sym^{s}(Z)$.  This map is dominant: since we are in the case when $s \leq 2d$, there exists a rational curve of degree $d$ through any $s$ points of $Z$.  By Lemma \ref{lemm:ratcurvesthroughpointsinp2deg11} the general fiber of this forgetful map is irreducible.  We deduce that $Q$, and hence $P$, is irreducible.  This finishes the proof of the claim for very free curves; it only remains to handle the curves that are free but not very free.  Since the desired statement is clear for multiple covers of fibers of $p$, we only need to consider the curves with vanishing intersection against $E_{0}$ and $E_{\infty}$.

\begin{lemm} \label{lemm:nonveryfreecurveson2E5}
Each of the numerical classes $R_{2}$, $2R_{2}$, and $3R_{2}$ is represented by a unique family of free curves.
\end{lemm}

\begin{proof}
First consider a family of free curves with numerical class $R_{2}$ or $2R_{2}$.  Under the birational map to $\mathbb{P}_{\mathbb{P}^{2}}(\mathcal{O} \oplus \mathcal{O}(2))$ the images of such curves will be lines or conics contained in a minimal moving section of the projective bundle.  Using this fact, we see that the image of such a curve $C$ under the projection map to $\mathbb{P}^{2}$ will be smooth and will meet the quartic curve lying under $Z$ transversally.  By applying Lemma \ref{lemm:ratcurvesthroughpointsinp2deg11} and arguing as above using the $\mathbb{C}^{*}$-action we deduce irreducibility in these two cases.

The argument for $3R_{2}$ is similar but requires more care with details.

\begin{lemm}
General fibers  of the evaluation morphism $ev: \overline{M}_{0,6} (\mathbb{P}^2,3)\to (\mathbb{P}^2)^{\times 6}$  
are irreducible. 
\end{lemm}

\begin{proof}
Consider the gluing morphism $gl: \overline{M}_{0,3}(\mathbb{P}^2,1) \times_{(\mathbb P^2)^2} \overline{M}_{0,5}(\mathbb{P}^2,2) \to  \overline{M}_{0,6} (\mathbb{P}^2,3)$. General fibers of $ev \circ gl$ are irreducible since the maps in the fiber over a general $6$-tuples $(p_1, \dots, p_6)$ are parametrized by 
conics through $p_3, \dots, p_6$. 
Any such map is a smooth point of 
$ev^{-1}(p_1, \dots, p_6)$, so general fibers of $ev$ are irreducible by  \cite[Lemma 3.2]{dejong-starr}.  \end{proof}
 
Let $I \subset \overline{M}_{0,6}(\mathbb{P}^2,3)$ be the fiber of the evaluation morphism over $Z^{\times 6}$. The dimension of every irreducible 
component of $I$ is $8$. If $p_1, \dots, p_6$ are such that no $4$ of them are on the same line then the fiber of $ev$ over 
 $(p_1, \dots, p_6)$ is $2$-dimensional. So if we denote by $U$ the open subset of $Z^{\times 6}$ of tuples of points such that no $4$ of them are on the same line, 
 then any irreducible component of $ev^{-1}(U)$ should map dominantly onto $U$. By the above lemma a general fiber of the evaluation map is irreducible, so $ev^{-1}(U)$ is irreducible.

Suppose $M$ is an irreducible component of $\overline{\Rat}(X)$ whose general points parametrize free irreducible curves of class $3R_2$.  If $C$ is a general curve parametrized by 
 $M$, then applying the $\mathbb{C}^*$-action to $C$ we get a curve parametrized by $M$ which is the union of an irreducible cubic $C'$ in $E_\infty$ and $6$ non-free lines intersecting $C'$ at the points of intersection of $C'$ with the quartic. We claim that the $6$ lines are disjoint.  Indeed, the only way that the $6$ lines could fail to be disjoint is if $C$ met the same fiber of $E$ at least twice so that the image of $C$ in $\bP_{\bP^{2}}(\mathcal{O} \oplus \mathcal{O}(2))$ is singular.  In particular, the image of $C$ under the map to $\bP^{6}$ cannot be a rational normal curve and so must be degenerate.  This implies that $C$ is contained in the strict transform in $X$ of a minimal moving section of $\bP_{\bP^{2}}(\mathcal{O} \oplus \mathcal{O}(2))$.  We then have the normal bundle sequence
$$0 \to N_{C/S} = \mathcal O(1) \to N_{C/X}  \to N_{S/X}|_C=\mathcal O(6) \to 0$$
where $S$ is the strict transform of the minimal moving section. This shows that $C$ is very free in $X$, a contradiction.

Since the six non-free lines intersecting $C'$ are distinct, the broken curve is a smooth point of $M$.  Since no $4$ points of the intersection can be on the same line, by the irreducibility result above we can conclude there is only one such component $M$.
\end{proof}

\begin{clai} \label{clai:mbbinhighdegree}
Let $X$ be as above.  Then Movable Bend-and-Break holds for free curves on $X$ of anticanonical degree at least $6$.
\end{clai}

We first prove the claim for curve classes of anticanonical degree between $6$ and $9$ and for the class $R_{3} + R_{4}$ of anticanonical degree $10$.  Suppose $\alpha$ is a nef class within this degree range.  Since $R_{1},\ldots,R_{6}$ form a $\mathbb{Z}$-generating set for the nef cone, we can write $\alpha$ as a sum of the $R_{i}$.  Thus by gluing a chain of free curves representing the $R_{i}$ we obtain a smooth point in the moduli space of stable maps of class $\alpha$.  A general point of the corresponding component will parametrize free curves, and thus by Claim \ref{clai:e5irreducibility} will  be the unique component of $\overline{M}_{0,0}(X)$ that generically parametrizes free curves of class $\alpha$.  After smoothing all but two components of the chain we find a point in this component representing a sum of two free curves.

We next prove the claim for curve classes of anticanonical degree at least $10$.  By Proposition \ref{prop:generalcurveclassification} every component of $\overline{M}_{0,0}(X)$ that generically parametrizes free curves in this degree range will either contain a union of two free curves or will contain a point $C_{1} \cup \ell \cup C_{2}$ parametrizing a union of two free curves connected by a line in an E5 divisor.  In the latter case, the two components $C_{1}$ and $C_{2}$ must be general in their deformation classes by construction.  Thus, Proposition \ref{prop:transverseintersection} guarantees that they intersect $E$ transversally.  We deduce that the stable map corresponding to our broken curve is a smooth point of $\overline{M}_{0,0}(X)$ by Proposition \ref{prop:vanishingofH1} and is thus contained in a unique component.  

In most cases at least one of the two free curves -- say $C_{1}$ -- will have anticanonical degree $\geq 6$.  Thus it can again be broken into a union of two free curves, yielding a curve of the form $(C_{1}' \cup C_{1}'') \cup \ell \cup C_{2}$ where $C_{1}' \cup C_{1}''$ is a general union of free curves obtain by applying Movable Bend-and-Break to $C_{1}$.  Note that this new curve underlies a smooth point of $\overline{M}_{0,0}(X)$.  By Proposition \ref{prop:vanishingofH1} we can smooth the subcurve $C_{1}'' \cup \ell \cup C_{2}$, and the resulting stable map will lie in our original component of $\overline{M}_{0,0}(X)$, verifying Movable Bend-and-Break for this component.

The remaining cases will be classes of anticanonical degree $10$ which break into a chain $C_{1} \cup \ell \cup C_{2}$ where $C_{1}$ and $C_{2}$ have anticanonical degree $4$ and $5$ respectively, or classes of anticanonical degree $11$ which break into a chain $C_{1} \cup \ell \cup C_{2}$ where $C_{1}$ and $C_{2}$ have anticanonical degree $5$.  If the curve of anticanonical degree $5$ lies in the class $R_{1} + R_{2}$ then we can argue similarly as before.  It only remains to consider the cases when all curves of anticanonical degree $5$ lie in $R_{3}$ or $R_{4}$.  Such curves can be deformed to the chain $C_{2}' \cup \ell' \cup C_{2}''$ where $C_{2}'$ and $C_{2}''$ are fibers of $p$ and $\ell'$ lies in the E5 divisor not containing $\ell$.  This broken curve with five components is a smooth point of $\overline{M}_{0,0}(X)$, and we can glue $C_{1} \cup \ell \cup C_{2}'$ to find a chain $C_{3} \cup \ell' \cup C_{2'}$ where $C_{3}$ has anticanonical degree $7$ or $8$.  We have then reduced to the situation where we can apply the earlier argument.

\begin{clai}
Let $X$ be as above.  Then there is a unique family of free curves representing any nef class $\alpha$.
\end{clai}

We have already verified the statement for curves of anticanonical degree $\leq 5$, so it suffices to prove the statement for anticanonical degree $\geq 6$.

Recall that each of $R_{1}, \ldots, R_{6}$ is represented by a unique family of free curves.  Since every nef class is a sum of these classes, we see that every nef class is represented by at least one family of free curves.

Conversely, we have already shown that each $R_{i}$ is represented by a unique irreducible family $M_{i}$ of free curves.  As above denote by $\mathcal{R}$ the commutative monoid of non-negative linear combinations of the formal symbols $R_{1},\ldots,R_{6}$.

Let $\mathcal{M}$ denote the components of $\overline{M}_{0,0}(X)$ which generically parametrize free curves.  By Theorem \ref{theo:iitakadim0casee5} the evaluation map for the universal family over each component $M \in \mathcal{M}$ has irreducible general fibers.  Thus if we identify two families of free curves on $X$ then by gluing and smoothing we can only obtain a unique component in $\mathcal{M}$.  By \cite[Lemma 5.11]{LTCompos} $\mathcal{M}$ admits the structure of an $\mathcal{R}$-module where the $\mathcal{R}$-action is given by gluing the corresponding curve classes and smoothing. 
By applying Movable Bend-and-Break in anticanonical degree $\geq 6$, we see that every component $M \in \mathcal{M}$ contains a point which is a chain of free curves from the components $M_{1},\ldots,M_{6}$.    In other words, there is a surjective $\mathcal{R}$-module homomorphism $\mathcal{R} \to \mathcal{M}$. 

As discussed above, two elements in $\mathcal{R}$ are identified under the numerical class homomorphism $\mathcal{R} \to \Nef_{1}(X)_{\mathbb{Z}}$ if and only if they can be identified using some sequence of relations while remaining in $\mathcal{R}$.  Each of these relations has anticanonical degree $\leq 9$ except for the relation $R_{1} + R_{5} + R_{6} = R_{3} + R_{4} $ in anticanonical degree $10$.  Claim \ref{clai:e5irreducibility} shows that there is a unique family of free rational curves parametrizing any of these classes, so we conclude that the surjection $\mathcal{R} \to \mathcal{M}$ factors through $\Nef_{1}(X)_{\mathbb{Z}}$.  In other words, there is at most one family representing any nef numerical class, finishing the proof of the claim.

\bibliographystyle{alpha}
\bibliography{MBB}

\end{document}